\documentclass[a4paper,twoside,10pt]{article}
\usepackage[a4paper,left=3cm,right=3cm, top=3cm, bottom=3cm]{geometry}
\usepackage[utf8]{inputenc}
\usepackage{mathrsfs}
\usepackage{graphicx}
\usepackage{epstopdf}
\usepackage{amsmath}
\usepackage{amsthm}
\usepackage[font=footnotesize,labelfont=bf]{caption}
\usepackage{algorithm,algpseudocode}
\usepackage{appendix}
\usepackage{amssymb}
\usepackage{stmaryrd}
\usepackage{cite}
\usepackage{xcolor}
\usepackage{pdfcomment}
\SetSymbolFont{stmry}{bold}{U}{stmry}{m}{n}

\usepackage{tikz}
\usetikzlibrary{calc}
\usetikzlibrary{arrows.meta,bending,positioning} 

\restylefloat{table}
\theoremstyle{plain}
\newtheorem{thm}{Theorem}[section]
\newtheorem{cor}[thm]{Corollary}
\newtheorem{lem}[thm]{Lemma}

\newtheorem{defn}[thm]{Definition}

\newtheorem{remark}[thm]{Remark}
\newtheorem{assumption}[thm]{Assumption}

\numberwithin{equation}{section}

\newcommand{\Norm}[2]{\left\Vert #1 \right\Vert_{#2}}
\newcommand{\SemiNorm}[2]{\left\vert #1 \right\vert_{#2}}
\let\div\relax
\DeclareMathOperator{\div}{div}

\newcommand{\sigmabold}{\boldsymbol \sigma}
\newcommand{\Nvert}{N_{{\rm vert}}}

\newcommand{\Ccontnu}{C_{{\rm cont},\nu}}
\newcommand{\Cstabnu}{C_{{\rm stab},\nu}}
\newcommand{\CcontE}{C_{{\rm cont},\E}}
\newcommand{\CstabE}{C_{{\rm stab},\E}}

\newcommand{\jump}[1]{\left[\!\left[#1\right]\!\right]}
\newcommand{\p}{p}
\newcommand{\h}{h}
\newcommand{\E}{K}
\newcommand{\SE}{S^\E}
\newcommand{\aE}{a^\E}
\newcommand{\ahE}{\aE_\h}
\newcommand{\uh}{u_\h}
\newcommand{\Pinablap}{\Pi^{\nabla}_\p}
\newcommand{\Scalh}{\mathcal S_\h}
\newcommand{\Gcalh}{\boldsymbol{\mathcal G}_\h}
\newcommand{\Nbb}{\mathbb N}
\newcommand{\Pbb}{\mathbb P}
\newcommand{\Rbb}{\mathbb R}
\newcommand{\RT}{\pmb{\mathbb{RT}}}
\newcommand{\Pbbs}{\pmb{\mathbb P}}
\newcommand{\nbf}{\mathbf n}

\newcommand{\nbfE}{\nbf^\E}
\newcommand{\Rcal}{\mathcal R}
\newcommand{\Nucalh}{\mathcal V_\h}
\newcommand{\omeganu}{\omega^\nu}
\newcommand{\rnu}{r^\nu}
\newcommand{\rnuh}{\rnu_\h}
\newcommand{\sigmaboldnu}{\boldsymbol \sigma^\nu}
\newcommand{\sigmaboldnuh}{\boldsymbol \sigma^\nu_\h}
\newcommand{\tauboldnu}{\boldsymbol \tau^\nu}
\newcommand{\tauboldnuh}{\boldsymbol \tau^\nu_\h}

\newcommand{\etanu}{\eta_\nu}
\newcommand{\etanuFLUX}{\eta_{\nu,{\rm FL}}}
\newcommand{\etanuPOT}{\eta_{\nu,{\rm PT}}}

\newcommand{\taun}{\mathcal T_\h}
\newcommand{\taunnu}{\taun^\nu}

\newcommand{\tautilden}{\widetilde{\mathcal T}_\h}
\newcommand{\tautildennu}{\tautilden^{\nu}}

\newcommand{\tautildeE}{\tautilden^\E}

\newcommand{\Ical}{\mathcal I}
\newcommand{\Vcaln}{\mathcal V_\h}
\newcommand{\VcalnB}{\Vcaln^B}
\newcommand{\VcalnI}{\Vcaln^I}
\newcommand{\Ecaln}{\mathcal E_\h}
\newcommand{\Ecalnu}{\mathcal E_{h,\nu}}
\newcommand{\EcalnuB}{\mathcal E_{h,\nu}^{\rm B}}
\newcommand{\EcalnB}{\Ecaln^B}
\newcommand{\EcalnI}{\Ecaln^I}
\newcommand{\xbfE}{\mathbf x_\E}
\newcommand{\hE}{\h_\E}

\newcommand{\VcalE}{\mathcal V^\E}
\newcommand{\EcalE}{\mathcal E^\E}
\newcommand{\e}{e}
\newcommand{\he}{\h_\e}

\let\div\relax
\DeclareMathOperator{\div}{div}
\newcommand{\Vh}{V_h}
\newcommand{\VhE}{\Vh(\E)}
\newcommand{\vh}{v_h}
\newcommand{\boldalpha}{\boldsymbol\alpha}
\newcommand{\mboldalpha}{m_{\boldalpha}}
\newcommand{\qp}{q_\p}
\newcommand{\qell}{q_\ell}
\newcommand{\qpmt}{q_{\p-2}}
\newcommand{\Piz}{\Pi^0_{\p-2}}
\newcommand{\ah}{a_\h}
\newcommand{\thetabold}{\boldsymbol \theta}
\newcommand{\thetaboldh}{\thetabold_\h}

\newcommand{\nablah}{\nabla_\h}

\newcommand{\phinu}{\varphi^\nu}

\newcommand{\qnu}{q^\nu}
\newcommand{\qnuh}{\qnu_h}

\newcommand{\snuh}{s_h^\nu}

\newcommand{\zh}{z_\h}

\newcommand{\snu}{s^\nu}
\newcommand{\etaE}{\eta_\E}
\newcommand{\ERR}{\mathcal E}

\newcommand{\omegaE}{\omega^\E}

\newcommand{\Gfrakh}{\mathfrak G_h}
\newcommand{\GfrakhE}{\Gfrakh^\E}
\newcommand{\Pizez}{\Pi^{0,\e}_0}
\newcommand{\Htildeonenu}{\widetilde H^1(\omeganu)}
\newcommand{\Honenu}{H^1(\omeganu)}
\newcommand{\btau}{\boldsymbol \tau}
\newcommand{\bG}{\boldsymbol G}
\newcommand{\bL}{\boldsymbol L}
\newcommand{\bH}{\boldsymbol H}
\newcommand{\bx}{\boldsymbol x}
\newcommand{\bn}{\boldsymbol n}
\newcommand{\ceff}{c_{\rm eff}}
\newcommand{\crel}{c_{\rm rel}}
\newcommand{\nbfe}{\nbf_\e}
\newcommand{\xp}{x_\p}
\newcommand{\Xp}{X_\p}
\newcommand{\XpE}{\Xp(\E)}

\author{\normalsize{Th\'eophile Chaumont-Frelet\thanks{
Inria Univ. Lille and Laboratoire Paul Painlev\'e, 59655 Villeneuve-d'Ascq, France,
{\tt theophile.chaumont@inria.fr}},
Joscha Gedicke\thanks{Institut f\"ur Numerische Simulation, Universit\"at Bonn, 53115 Bonn,
{\tt gedicke@ins.uni-bonn.de}},
Lorenzo Mascotto\thanks{Dipartimento di Matematica e Applicazioni, Universit\`a di Milano Bicocca, 20125 Milan, Italy, {\tt lorenzo.mascotto@unimib.it};
IMATI-CNR, 27100, Pavia, Italy;
Fakult\"at f\"ur Mathematik, Universit\"at Wien, 1090 Vienna, Austria}}}{}
\date{}

\date{}
\title{\normalsize{Generalised gradients for virtual elements and applications to a posteriori error analysis}}{}

\begin{document}

\maketitle
\begin{abstract}
\noindent
We rewrite the standard nodal virtual element method
in two dimensions
as a generalised gradient method.
This re-formulation allows for computing
a reliable and efficient error estimator
by locally reconstructing broken fluxes and potentials
on elemental subtriangulations.
We prove upper and lower bounds with constants independent of the stabilisation of the method
and, under technical assumptions on the mesh,
the degree of accuracy.
\medskip

\medskip
		
\end{abstract}

\section{Introduction} \label{section:introduction}

The virtual element method (VEM) \cite{Beirao-Brezzi-Cangiani-Manzini-Marini-Russo:2013}
is a generalisation of the finite element method (FEM)
to polytopic meshes whose elements may have
curved and/or hanging facets, internal cuts, and nonconvex shapes.
Virtual element spaces typically consist of functions that
solve local variational problems with polynomial data
and are not required to be known in closed form,
but rather via their evaluation through suitable degrees of freedom.
Such degrees of freedom are selected so as to permit
the computation of projection operators onto polynomial spaces,
which are inserted in the bilinear form
in order to preserve the polynomial consistency of the scheme.
The well posedness of the method is guaranteed by
supplementing the projected bilinear form with a stabilisation,
which has to be computable via the degrees of freedom.

One among the appealing features of the scheme
is that adaptive mesh refinements admit the creation of hanging facets
without the need of generating extra (useless) elements.
The literature on the adaptive VEM is rather wide.
We mention here only the early contributions~\cite{BeiraodaVeiga-Manzini:2015, Cangiani-Georgulis-Pryer-Sutton:2017, Berrone-Borio:2017}.
Almost all the related references are concerned with residual-type error estimators,
which present two major downsides:
\begin{itemize}
    \item the constants in the reliability and efficiency bounds depend on the choice of the stabilisation;
    \item as in the finite element case~\cite{Melenk-Wohlmuth:2001},
    the efficiency constant depends on the degree of accuracy of the method~\cite{BeiraodaVeiga-Manzini-Mascotto:2019}.
\end{itemize}
The first issue becomes relevant when performing aggressive mesh refinement and/or coarsening:
ad-hoc stabilisations should be designed so as to derive a posteriori
error bounds, which are robust with respect to the presence of
small facets, highly distorted elements, and so on.
Partial improvements are discussed in~\cite{Beirao-Canuto-Nochetto-Vacca-Verani:2023}
where the upper bound is independent of the stabilisation
for lowest order elements on triangular grids.

Using local flux equilibration strategies,
$\p$-robust error estimators were introduced
for the FEM and discontinuous Galerkin methods;
see, e.g., \cite{Braess-Pillwein-Schoeberl:2009, Ern-Vohralik:2020}.
Inspired by the above works,
a first attempt in deriving a $\p$-robust error estimator for the VEM
was performed in~\cite{Dassi-Gedicke-Mascotto:2022}.
There, a $\p$-robust error estimator was derived
based on the computation of a virtual \emph{global} flux.
However, straightforwardly adapting the \emph{localisation}
procedure in~\cite{Destuynder-Metivet:1999,Ern-Vohralik:2015} to the VEM
led to the loss of the efficiency,
one of the main reasons being the lack of Galerkin orthogonality
for the gradient of the discrete solution.

This motivated us to rewrite in this contribution the VEM as 
a generalised gradient method.
This amounts to introducing a post-processing of
the discrete solution into a generalised gradient
that exhibits favourable orthogonality properties.
More precisely, on each element~$\E$,
we rewrite the usual virtual element discrete bilinear form,
which typically looks like
\[
\ahE(\cdot,\cdot) 
= \aE(\Pi \cdot, \Pi \cdot) 
    + \SE( (I-\Pi) \cdot, (I-\Pi)\cdot),
\]
$\Pi$ being a computable polynomial projector
and~$\SE(\cdot,\cdot)$ a computable stabilisation,
as a generalised gradient bilinear form.
In particular, we shall write
\[
\ahE(\cdot,\cdot) = (\GfrakhE(\cdot), \nabla \cdot)_{\E},
\]
where~$\GfrakhE(\cdot)$ is a generalised gradient
given by a piecewise Raviart-Thomas function over a subtriangulation of~$\E$,
which contains information on the stabilisation.
This generalised gradient is available for \emph{any} computable
stabilisation and is obtained by solving local finite element problems.
These problems can be independently solved in parallel
and their solution is cheap.
The subtriangulations mentioned above
are always used in practice on the elemental level;
as such, the approach presented in this paper
can be interesting per se while compared to a finite element one,
where typically elements are refined using newest-vertex bisection
and possibly neighbouring elements are to be refined to maintain
the conformity of the method.

Even if the focus of the paper is on conforming virtual elements
for the Poisson problem in two dimensions,
extensions to other settings are possible.
For instance, nonconforming virtual elements
or virtual elements in three dimensions (with simplicial facets)
in principle fit to the setting described above.
In fact, the construction of a generalised gradient
in the context of the virtual element
is based on only two ingredients:
a projection onto piecewise polynomials
and a lifting of the stabilisation term.

The use of a generalised gradient is also instrumental in the design 
of a posteriori error estimators
for interior penalty discontinuous Galerkin schemes on simplicial and tensor product meshes,
which also suffer from a lack of Galerkin orthogonality caused by the stabilisation~\cite{Ern-Vohralik:2015}.
A similar construction was also used in~\cite{DiPietro-Droniou-Manzini:2018}
in the framework of approximation of solutions to
nonlinear problems for skeletal methods.

This generalised gradient allows for the design of an error estimator
that is reliable and efficient with bounds independent of the chosen stabilisation
as well as of the degree of accuracy of the method;
see Theorem~\ref{theorem:reliability-efficiency}.
Specifically, once the solution~$u_h$ and the generalised gradient
$\Gfrakh(u_h)$ are known,
the computation of the error estimator involves the solution
to local primal and mixed finite element problems
on vertex patches as defined
in~\eqref{eq:definition_etanu} below.

As such, the implementation of the error estimator can be made fully parallel
and no more expensive than computing a residual error estimator.
As we detail in Remark~\ref{remark:comparison_DG} below,
the constants involved in the upper and lower bounds
generalise those obtained in the corresponding bounds
for discontinuous Galerkin schemes in~\cite{Ern-Vohralik:2015}
to polytopic meshes:
the resulting estimates are as good as the state-of-the-art
for another nonconforming numerical method on standard meshes,
yet keeping in mind that a modified error measure is considered.

Numerical examples show that the generalised gradient
provides us with an approximation of the exact solution
comparable to that obtained with a polynomial
energy projection of the discrete solution,
i.e., the standard error measure in virtual elements.
These examples also illustrate that the estimator provides a fairly
sharp upper bound on the error that does not deteriorates
as $\p$ increases.

\paragraph*{Outline of the paper.}

We introduce the setting and the necessary notation in Section~\ref{section:setting}.
We introduce the VEM and an equivalent rewriting
using a generalised gradient
in Section~\ref{section:VEM}.
Some technical results are detailed in Section~\ref{section:technical-results}.
In Section~\ref{section:eq:prager_synge},
we introduce a computable error estimator and state its key properties.
The estimator consists of two terms,
analysed in Sections~\ref{section:flux} and~\ref{section:potential}, respectively.
We present numerical examples in Section~\ref{section:numerical-results}
and draw some conclusions
in Section~\ref{section:conclusions}.
Finally, Appendix~\ref{appendix:discrete-min} presents technical results,
which are related to the $\p$-robustness of the efficiency estimate.

\paragraph*{Main result.}
We introduce a generalised gradient for the VEM in Definition~\ref{defn:gg}
and show that it can be used to equivalently rewrite the VEM in Theorem~\ref{thm:vem-rewriting}.
The error estimator is introduced in~\eqref{eq:definition_etanu}, and its main properties are
given in Theorem~\ref{theorem:reliability-efficiency}
under Assumption~\ref{assumption:galerkin-orthogonality}.

\section{Setting} \label{section:setting}
In Section~\ref{subsection:functional-spaces}, we introduce the functional spaces that we shall employ throughout.
The model problem is given in Section~\ref{subsection:continuous-problem},
while polygonal meshes and their assumptions are detailed in Section~\ref{subsection:meshes}.
Several polynomial spaces are recalled in Section~\ref{subsection:polynomials}.

\subsection{Functional spaces} \label{subsection:functional-spaces}
Consider an open set $D$ of~$\Rbb^2$.
For $q$ in $[1,\infty]$, we employ the standard notation~$L^q(D)$
for Lebesgue spaces equipped with their usual norm $\|{\cdot}\|_{L^q(D)}$.
When $q=2$, $L^2(D)$ is a Hilbert space, and its norm and inner-product
are denoted by $\|{\cdot}\|_D$ and $(\cdot,\cdot)_D$.
For vector-valued functions, we set~$\bL^2(D) := [L^2(D)]^2$,
and use the same notation for its norm and inner-product.
$H^1(D)$ and $\bH(\div,D)$ stand for the Sobolev spaces
of scalar-valued functions~$v$ in~$L^2(D)$
such that~$\nabla v$ belongs to~$\bL^2(D)$
and vector-valued functions~$\btau$ in~$\bL^2(D)$
with~$\div \btau$ in $L^2(D)$, respectively;
we endow the space $H^1(D)$
with its norm, seminorm, and inner-product
$\Norm{\cdot}{1,D}$, $\SemiNorm{\cdot}{1,D}$, and $(\cdot,\cdot)_{1,D}$.
$H^1_0(D)$ is the closure in $H^1(D)$ of the space
of smooth functions with compact support in~$D$.
If $\mathcal D$ is a finite collection of disjoint open sets,
then $H^1(\mathcal D)$ collects functions~$v$
such that $v|_D$ is in~$H^1(D)$ for all $D$ in $\mathcal D$.

\subsection{The continuous problem}\label{subsection:continuous-problem}
We consider a fixed Lipschitz polygonal domain~$\Omega$ in $\Rbb^2$
and $f$ in $L^2(\Omega)$.
The model problem we are interested in consists in finding $u$ in~$H^1_0(\Omega)$ such that
\begin{equation} \label{weak-formulation}
(\nabla u, \nabla v)_{\Omega} = (f,v)_{\Omega}
\qquad\qquad\qquad\qquad
\forall v \in H^1_0(\Omega) .
\end{equation}
Standard arguments reveal that~\eqref{weak-formulation} is well posed.
In what follows, $u$ will always denote the solution to~\eqref{weak-formulation}.

\subsection{Regular polygonal meshes} \label{subsection:meshes}
We consider meshes~$\taun$ consisting of open, conforming
\footnote{Conforming here means that the meshes consist of
open simply connected polygons whose boundary is a nonintersecting
line made of a finite number of straight line segments.}
polygonal elements over~$\Omega$.
We denote their set of vertices and facets by~$\Vcaln$ and~$\Ecaln$.
To each element~$\E$ of~$\taun$, we associate its diameter~$\hE$, centroid~$\xbfE$,
set of vertices~$\VcalE$, and set of facets~$\EcalE$.
We introduce the mesh size~$\h := \max_{\E\in\taun}\hE$
and the length~$\he$ of each facet~$\e$ in~$\Ecaln$.
We further split internal and boundary vertices into the sets~$\VcalnI$ and~$\VcalnB$,
and internal and boundary facets into the sets~$\EcalnI$ and~$\EcalnB$.
We denote the maximal number of vertices of the elements of~$\taun$
by $\Nvert$.

Henceforth, we assume the following regularity properties to be valid:
there exists~$\gamma$ in $(0,1]$ such that
\begin{itemize}
    \item each element~$\E$ is star-shaped with respect to a ball~$B(\E)$ of radius larger than or equal to~$\gamma \hE$;
    \item given an element~$\E$, any of its facets~$\e$ has length~$\he$ larger than~$\gamma\hE$.
\end{itemize}
In what follows, $C(\gamma)$ denotes a generic positive constant,
which only depends on the parameter~$\gamma$.

The two above properties imply that each element~$\E$ has a uniformly bounded number of
vertices. Moreover, associated with each~$\E$,
it is possible to construct a shape-regular subtriangulation~$\tautildeE$
by connecting the centre of the ball~$B(\E)$
with the vertices of~$\E$;
see, e.g., \cite{Chen-Huang:2018}.
Using the regularity properties of~$\taun$,
we deduce the regularity of~$\tautildeE$.

We define~$\tautilden$ as the union of all~$\tautildeE$;
$\tautildennu$ denotes the set of triangles
contained in elements of~$\taun$ that share the vertex~$\nu$.
Henceforth, we assume that such a subtriangulation $\tautilden$ is fixed.
The broken gradient over either~$\taun$ or~$\tautilden$ is denoted by~$\nablah$.
\medskip

Only while deriving $\p$-robust efficiency bounds,
we shall require the further technical assumption on the mesh.

\begin{assumption}[Single facet sharing]
\label{assumption_faces}
\begin{subequations}
Any two distinct elements $\E$ and~$\E'$ in~$\taun$ share at most one facet, i.e.,
\begin{equation}
\label{assumption_interior_faces}
\partial K \cap \partial K' \in \Ecaln 
\qquad\qquad
\forall K,K' \in \taun, \; K \neq K', \; \partial K \cap \partial K' \neq \emptyset.
\end{equation}
For any boundary elements, we further have that
\begin{equation} \label{assumption_exterior_faces}
\partial K \cap \partial \Omega \in \Ecaln 
\qquad\qquad
\forall K \in \taun, \; \partial K \cap \partial \Omega \neq \emptyset.
\end{equation}
\end{subequations}
\end{assumption}

\begin{figure}
\centering
\begin{minipage}{.49\linewidth}
\centering
\begin{tikzpicture}[scale=2.5]

\coordinate (a)   at (0,0);

\coordinate (a0)  at ($(a) + (-10:0.9)$);
\coordinate (a1)  at ($(a) + (100:1.1)$);
\coordinate (a2)  at ($(a) + (-110:0.8)$);

\coordinate (a01) at ($(a) + (45:1.4)$);

\coordinate (a11) at ($(a) + (160:1.4)$);
\coordinate (a12) at ($(a) + (190:1.4)$);
\coordinate (a13) at ($(a) + (220:1.4)$);

\coordinate (a21) at ($(a) + (-70:1.2)$);
\coordinate (a22) at ($(a) + (-30:1.4)$);

\draw[ultra thick] (a) -- (a0);
\draw[ultra thick] (a) -- (a1);
\draw[ultra thick] (a) -- (a2);

\draw[ultra thick] (a0) -- (a01) -- (a1);
\draw[ultra thick] (a1) -- (a11) -- (a12) -- (a13) -- (a2);
\draw[ultra thick] (a2) -- (a21) -- (a22) -- (a0);

\fill (a)   circle (0.05);

\fill (a0)  circle (0.05);
\fill (a1)  circle (0.05);
\fill (a2)  circle (0.05);

\fill (a01) circle (0.05);

\fill (a11) circle (0.05);
\fill (a12) circle (0.05);
\fill (a13) circle (0.05);

\fill (a21) circle (0.05);
\fill (a22) circle (0.05);

\end{tikzpicture}
\end{minipage}
\begin{minipage}{.49\linewidth}
\centering
\begin{tikzpicture}[scale=2.5]

\coordinate (a) at (0,0);

\coordinate (a0) at ($(a) + (-10:0.9)$);
\coordinate (a1) at ($(a) + (100:1.1)$);
\coordinate (a2) at ($(a) + (-110:0.8)$);

\coordinate (a01) at ($(a) + (45:1.4)$);

\coordinate (a11) at ($(a) + (160:1.4)$);
\coordinate (a12) at ($(a) + (190:1.4)$);
\coordinate (a13) at ($(a) + (220:1.4)$);

\coordinate (a21) at ($(a) + (-70:1.2)$);
\coordinate (a22) at ($(a) + (-30:1.4)$);

\coordinate (b) at ($(a) + (-140:0.5)$);

\draw[ultra thick] (a) -- (a0);
\draw[ultra thick] (a) -- (a1);
\draw[ultra thick] (a) -- (b) -- (a2);

\draw[ultra thick] (a0) -- (a01) -- (a1);
\draw[ultra thick] (a1) -- (a11) -- (a12) -- (a13) -- (a2);
\draw[ultra thick] (a2) -- (a21) -- (a22) -- (a0);

\fill (a)   circle (0.05);

\fill (a0)  circle (0.05);
\fill (a1)  circle (0.05);
\fill (a2)  circle (0.05);

\fill (a01) circle (0.05);

\fill (a11) circle (0.05);
\fill (a12) circle (0.05);
\fill (a13) circle (0.05);

\fill (a21) circle (0.05);
\fill (a22) circle (0.05);

\fill (b)   circle (0.05);

\end{tikzpicture}
\end{minipage}
\caption{Examples of configurations allowed (\emph{left panel})
and forbidden (\emph{right panel})
under Assumption~\ref{assumption_faces}
for the derivation of $\p$-robust estimates.
The $\h$-version allows for both configurations.}
\label{figure_configurations_faces}
\end{figure}
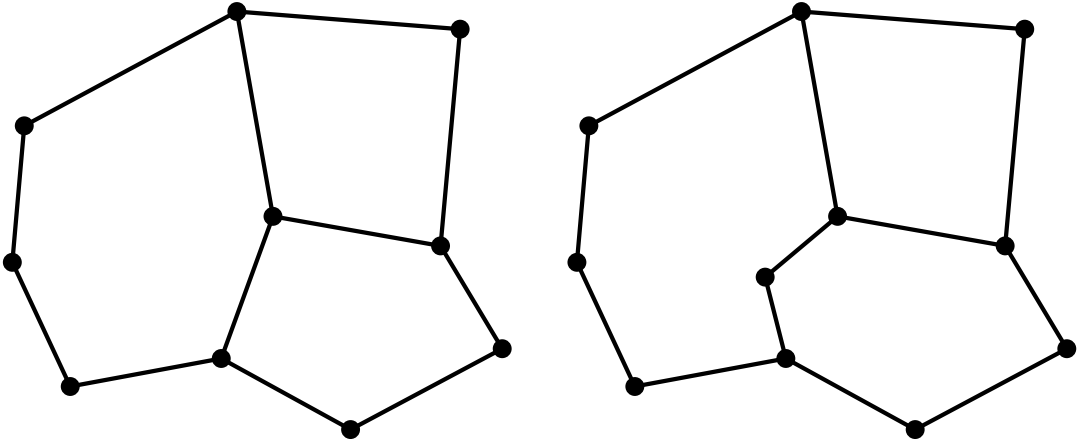

\begin{remark}
In principle, we may also consider more general
(and not necessarily star-shaped) polygonal elements
as long as we can partition them into shape-regular subtriangulations
with mesh size comparable to the diameter of the element.
The existence of a subtriangulation is needed not only
as an abstract device to measure the regularity of the polygonal mesh
but also used in practice in the computation of the generalised gradient.
Similar to the case of meshes not satisfying Assumption~\ref{assumption_faces},
the $\p$-robustness remains an open question.
\end{remark}

\subsection{Standard polynomial spaces} \label{subsection:polynomials}
If~$D$ in~$\mathbb R^2$ is an open set
and~$q$ is in~$\Nbb$,
$\Pbb_q(D)$ stands for the space of polynomials of degree
at most~$q$ over~$D$.
$\Pbbs_q(D) := [\Pbb_q(D)]^2$ contains vector-valued polynomials
and $\RT_q(D) := \Pbbs_q(D) + \bx \Pbb_q(D)$
is the set of Raviart--Thomas polynomials.
As before, if $\mathcal D$ is a collection of disjoint elements,
then $\Pbb_q(\mathcal D)$, $\Pbbs_q(\mathcal D)$,
and $\RT_q(\mathcal D)$ consist of piecewise polynomial functions
over that collection.
We also use the notation $\mathbb P_{-1}(D) = \{0\}$.

\section{The virtual element method and generalised gradients} \label{section:VEM}
We present two equivalent versions of the VEM.
In Section~\ref{subsection:VE-spaces}, we introduce the virtual element spaces
with their degrees of freedom,
which allow us to compute the stabilisations
and projection operators
detailed in Section~\ref{subsection:projections-stabilisations}.
The standard VEM is presented in Section~\ref{subsection:VEM},
while an equivalent rewriting based on a generalised gradient
is given in Section~\ref{subsection:gg-rewriting-of-VEM}.

\subsection{Virtual element spaces} \label{subsection:VE-spaces}
We define~\cite{Beirao-Brezzi-Cangiani-Manzini-Marini-Russo:2013}
the virtual element space of order~$\p$ in~$\Nbb$ on the element~$\E$ as
\[
\VhE:=
\left\{
\vh \in H^1(\E) \mid
\Delta\vh \in \Pbb_{\p-2}(\E),\;\;
\vh{}_{|\partial \E} \in \mathcal C^0(\partial \E),\;
\vh{}_{|\e} \in \Pbb_\p(\e) \ \forall \e \in \EcalE \right\}.
\]
We endow~$\VhE$ with the following unisolvent degrees of freedom~\cite{Beirao-Brezzi-Cangiani-Manzini-Marini-Russo:2013}:
given~$\vh$ in~$\VhE$,
\begin{itemize}
    \item the point values of~$\vh$ at the vertices~$\nu$ in~$\VcalE$;
    \item for each facet~$\e$ of~$\E$, the point values of~$\vh$ at the $\p-1$ internal Gau\ss-Lobatto nodes of~$\e$;
    \item given~$\{ \mboldalpha \}$ a basis of~$\Pbb_{\p-2}(\E)$
    of elements that are invariant
    with respect to dilations and translations
    \footnote{A standard basis satisfying the properties
    above is given by that of scaled monomials
    as detailed, e.g., in \cite{Beirao-Brezzi-Cangiani-Manzini-Marini-Russo:2013}.},
    the scaled moments
    \[
    \vert\E\vert^{-1} \int_\E \vh \ \mboldalpha
    \qquad\qquad\qquad\qquad \forall \vert \boldalpha \vert=0,\dots,\p-2.
    \]
\end{itemize}
The corresponding $H^1$ conforming global space reads
\[
\Vh := \{ \vh \in H^1(\Omega) \mid \vh{}_{|\E} \in \VhE  \}.
\]
A set of unisolvent degrees of freedom for~$\Vh$ is given by the $H^1$ conforming coupling of the degrees of freedom of each~$\VhE$.

\subsection{Stabilisations and polynomial projectors} \label{subsection:projections-stabilisations}
Consider the splitting
\[
(\nabla u,\nabla v) = \sum_{\E \in \taun} (\nabla u, \nabla v)_{\E}.
\]
On each element~$\E$, we introduce a bilinear form~$\SE(\cdot,\cdot)$ satisfying three properties:
(\emph{i}) $\SE(1,1)>0$;
(\emph{ii}) $\SE(\cdot,\cdot)$ is computable via the degrees of freedom of~$\VhE$;
(\emph{iii}) there exist $0 < \alpha_* < \alpha^*$ such that
\[
\alpha_* \SemiNorm{\vh}{1,\E}^2
\le \SE(\vh, \vh)
\le 
\alpha^* \SemiNorm{\vh}{1,\E}^2
\qquad\qquad\qquad
\forall \vh \in \VhE,\quad (\vh,1)_\E =0.
\]
The degrees of freedom of~$\VhE$ allow for the computation of projection operators onto polynomial spaces~\cite{Beirao-Brezzi-Cangiani-Manzini-Marini-Russo:2013}.
We define~$\Pinablap:\VhE \to \Pbb_\p(\E)$ as
\begin{equation} \label{Pinabla}
(\nabla(\vh-\Pinablap \vh),\nabla \qp) = 0,
\quad
\SE(\vh- \Pinablap \vh,1) = 0
\qquad\qquad \forall \vh \in \VhE,\; \qp \in \Pbb_\p(\E).
\end{equation}
The second condition above fixes constants
due to the assumption~$\SE(1,1)>0$.
Differently from the standard virtual element setting,
we fix the average of~$\Pinablap \vh$
through the stabilisation~$\SE(\cdot,\cdot)$.
This will be crucial while rewriting the method
as a generalised gradient method
in Section~\ref{subsection:gg-rewriting-of-VEM} below.

We define~$\Piz:\VhE \to \Pbb_{\p-2}(\E)$ as
\begin{equation} \label{Piz}
(\vh - \Piz \vh, \qpmt)_{\E} = 0
\qquad\qquad \forall \vh \in \VhE,\; \qpmt \in\Pbb_{\p-2}(\E).
\end{equation}
In what follows, with an abuse of notation,
we introduce global projection operators~$\Pinablap$ and~$\Piz$
onto piecewise discontinuous polynomial spaces over~$\taun$
so that their restrictions to each element
coincide with the operators in~\eqref{Pinabla}
and~\eqref{Piz}, respectively.

We introduce the local and global bilinear forms
\[
\begin{split}
& \ahE(\phi_h,v_h) :=
(\nabla(\Pinablap \phi_h),\nabla(\Pinablap v_h))
+ \SE( (I-\Pinablap)\phi_h, (I-\Pinablap)v_h),\\
& \ah(\phi_h,v_h):= \sum_{\E \in \taun} \ahE(\phi_h,v_h)
\qquad\qquad \forall \phi_h, v_h \in \Vh.
\end{split}
\]
In particular, the global discrete bilinear form~$\ah(\cdot,\cdot)$
is coercive and continuous with respect to the energy norm
with constants~$\min(1,\alpha_*)$ and~$\max(1,\alpha^*)$,
respectively.

\begin{remark} \label{remark:explicit-stabilisations}
Several explicit choices for the stabilisation~$\SE(\cdot,\cdot)$ are available
in the literature~\cite{Mascotto:2023}.
Amongst others,
we mention the so-called ``dofi-dofi'' stabilisation
given by the $\ell^2$ product of the degrees of freedom,
and the ``projected'' stabilisation
\begin{equation} \label{projected-stab}
\SE(\phi_h,v_h) := \hE^{-2} (\Piz \phi_h, \Piz v_h)_{\E}
                 + \hE^{-1} (\phi_h, v_h)_{\partial \E} .
\end{equation}
At any rate, the forthcoming analysis is fairly general and works for any choice of~$\SE(\cdot,\cdot)$
satisfying the three properties above.
\end{remark}

\subsection{The standard formulation of the VEM}
\label{subsection:VEM}
We only consider the case where $\p$ is larger than or equal to~$2$.
The case~$\p$ equal to~$1$
can be analogously handled with modifications
as in~\cite{Beirao-Brezzi-Cangiani-Manzini-Marini-Russo:2013}.
Then, a virtual element method for the approximation of the solution
to~\eqref{weak-formulation} reads
\begin{equation} \label{VEM}
    \begin{cases}
        \text{find } \uh \in \Vh \text{ such that}\\
        \ah (\uh,\vh) = (f,\Piz \vh)_{\Omega} \qquad \forall \vh \in \Vh.
    \end{cases}
\end{equation}
Method~\eqref{VEM} is well posed due to the continuity and coercivity of~$\ah(\cdot,\cdot)$,
and the continuity of~$(f,\Piz \cdot)_{\Omega}$.
In what follows, $\uh$ in~$\Vh$ will always denote the solution to~\eqref{VEM}.

To simplify the presentation, we henceforth assume that~$f$ belongs to~$\Pbb_{\p-2}(\taun)$.
This allows us to remove high-order data oscillation terms in the forthcoming bounds.
The analysis of the general case poses no extra difficulties but would result in a more
cumbersome notation.
In particular, the a posteriori error estimates
in Theorem~\ref{theorem:reliability-efficiency} below
would contain extra, higher order terms of the form
\[
\sum_{\E\in\taun} \hE^2 \Norm{f-\Piz f}{0,\E}^2
\]
for the upper bound, and a similar sum over the elements
on certain patches for the lower bound.

\subsection{The VEM as a generalised gradient method}
\label{subsection:gg-rewriting-of-VEM}
We rewrite method~\eqref{VEM} as a generalised gradient method.
In other words, we describe how to design a suitable, fully-computable
operator $\Gfrakh: V_h \to \bL^2(\Omega)$ such that
\begin{equation}
\label{eq:generalised-gradient}
\ah(\phi_h,v_h) = ( \Gfrakh(\phi_h), \nabla v_h)_{\Omega}
\qquad\qquad\qquad\qquad
\forall \phi_h, v_h \in \Vh.
\end{equation}
We propose a construction that is fully-localised to each element, so that
$\Gfrakh(\phi_h)_{|\E}$ belongs to $\RT_\p(\tautildeE)$
and~\eqref{eq:generalised-gradient}
holds true element-wise.

A key ingredient of the generalised gradient reconstruction
is a ``stabilisation lifting'': for the explicit stabilisation
\begin{equation*}
\SE(\cdot,\cdot) := \hE^{-2} (\Piz\cdot, \Piz\cdot)_{\E} + \hE^{-1} (\cdot,\cdot)_{\partial \E},
\end{equation*}
we can write
\begin{equation} \label{write-explicit-stab}
\SE((I-\Pinablap)\phi_h,v_h)
= (\mu_h^K(\phi_h),v_h)_{\partial \E}-(r_h^K(\phi_h),v_h)_\E,
\end{equation}
where we have set
\begin{equation*}
\mu_h^K(\phi_h) 
:= \hE^{-1} (I-\Pinablap) \phi_h{}_{|\partial\E}
            \in \Pbb_{\p}(\EcalE),
\quad\quad
r_h^K(\phi_h) := -\hE^{-2} \Piz(I-\Pinablap) \phi_h \in \Pbb_{\p-2}(\E).
\end{equation*}
We show that it is \emph{always}
possible to rewrite the stabilisation as in~\eqref{write-explicit-stab},
and that the data~$\mu_h^K(\phi_h)$ and~$r_h^K(\phi_h)$
can be found by solving a local linear system.

\begin{lem}[Stabilisation lifting]
\label{lem:stab_lift}
For all $\phi_h$ in~$\VhE$, there exist unique~$\mu_h^K(\phi_h)$ in~$\Pbb_{\p}(\EcalE)$
and~$r_h^K(\phi_h)$ in~$\Pbb_{\p-2}(\E)$ such that
\begin{equation} \label{identity:stab-lifting}
(\mu_h^K(\phi_\h), \vh)_{\partial \E}
    -(r_h^K(\phi_\h),\vh)_\E
= \SE((I-\Pinablap) \phi_\h,\vh)
    \qquad\qquad \forall \vh \in \VhE.
\end{equation}
\end{lem}

\begin{proof}
A practical construction of~$(r_\h^\E,\mu_\h^\E)$ is as follows.
Consider the space~$\Pbb_\p(\EcalE)$
of piecewise continuous polynomials
of maximum degree~$\p$ over~$\partial\E$.
Define $\XpE:=\Pbb_\p(\EcalE) \times \Pbb_{\p-2}(\E)$
and
\[
\Lambda : \XpE \times \VhE  \to \Rbb,
\qquad\qquad\qquad
\Lambda(\xp,\vh)
:= (\xp^{\partial\E}, \vh)_{\partial\E}
   - (\xp^{\E}, \vh)_{\E}.
\]
Consider the canonical basis $\{ \varphi_j \}_{j=1}^{\dim(\VhE)}$ of~$\VhE$,
which is dual to the degrees of freedom in Section~\ref{subsection:VE-spaces}.
This basis, and a (polynomial) basis $\{ \qell \}_{\ell=1}^{\dim(\VhE)}$,
$\qell=(\qell^{\partial\E}, \qell^\E)$,
of~$\XpE$ have the same cardinality.
In fact, the numbers of boundary and bulk degrees of freedom
detailed in Section~\ref{subsection:VE-spaces}
match the dimensions of the spaces $\Pbb_\p(\EcalE)$ and $\Pbb_{\p-2}(\E)$,
respectively.
In particular, there exists a natural bijection between
the spaces~$\XpE$ and~$\VhE$,
which we denote with an abuse of notation by $\phi_\h(\xp)$.

We have to prove that there exists a unique element~$\xp$ in~$\XpE$
such that
\[
\Lambda(\xp,\vh)
= \SE((I-\Pinablap) \phi_\h(\xp), \vh)
    \qquad\qquad\qquad \forall \vh \in \VhE.
\]
It suffices to solve the square system
with coefficient matrix and right-hand side given by
\small{\[
\mathbf A_{i,j}
:= (q_j^{\partial\E},\varphi_i)_{\partial \E} 
    - (q_j^\E,\varphi_i)_{\E},
\quad 
\mathbf b_{i} := \SE((I-\Pinablap)\phi_h, \varphi_{i})
    \qquad \forall i,j=1,\dots,\dim(\VhE).
\]}\normalsize
The system is well posed. In fact, given any~$y_\p$ in $\XpE$,
there exists a unique~$\vh$ in~$\VhE$
such that the trace of~$\vh$ on $\partial\E$
is $y_\p^{\partial\E}$
and $\Piz\vh$ equals $-y_\p^\E$.
This implies that if we are given~$\xp$ in $\XpE$
such that $\Lambda(\xp, w_\h) = 0$ for all~$w_\h$ in~$\VhE$,
then $\xp$ is zero. To see this, we write
\[
0 = \Lambda(\xp,\vh)
  = \Norm{\xp^{\partial\E}}{\partial\E}^2
    + \Norm{\xp^\E}{\E}^2.
\]
We proved that the kernel of $\Lambda$ is zero.
Such a kernel coincides with that of the square matrix $\mathbf A$,
which is consequently invertible.
We find~$r_\h^K$ and~$\mu_h^K$ by considering
linear combinations with respect to the polynomial basis~$\qell$
and coefficients given by the solution to the above linear system.
\end{proof}

Next, we introduce the generalised gradient operator.
\begin{defn}[Generalised gradient]
\label{defn:gg}
For all $K$ in~$\taun$ and $\phi_h$ in~$\VhE$, we define
\begin{equation} \label{generalised-gradient}
\GfrakhE(\phi_h)
:= \nabla(\Pinablap \phi_h-\Scalh^K(\phi_h)) + \thetaboldh^K(\phi_h)
\in \RT_\p(\tautildeE) ,
\end{equation}
where $\Scalh^K(\phi_h)$ in~$\Pbb_\p(K)$ satisfies
\begin{equation} \label{definition:Scalh}
(\nabla \Scalh^K(\phi_h),\nabla q_h)_K = \SE((I-\Pinablap) \phi_h,q_h)
\qquad\qquad \forall q_h \in \Pbb_\p(K)
\end{equation}
and
\footnote{The computation of~$\thetaboldh$ below is equivalent to
solving a local mixed finite element method on the element~$\E$.}
\begin{equation}  \label{def:thetabold}
\thetaboldh^K(\phi_h)
:= \arg \min_{\substack{
\btau_h \in \RT_{\p+1}(\tautildeE) \cap \bH(\div,\E) \\
 \div \btau_h = r_h^K(\phi_h) \text{ in $\E$} \\
\btau_h \cdot \bn = \mu_h^K(\phi_h) \text{ on $\partial\E$}}}
\Norm{\btau_h}{\E}.
\end{equation}
If $\phi_h$ belongs to~$\Vh$,
then we define $\Gfrakh(\phi_h)$ in $\RT_{\p}(\tautilden)$
by setting $\Gfrakh(\phi_h)_{|\E}
:= \GfrakhE(\phi_h{}_{|\E})$.
\end{defn}

The compatibility conditions for the definition
of~$\thetaboldh^K$ in~\eqref{def:thetabold}
are guaranteed by the way we fix the constant part of~$\Pinablap$
in~\eqref{Pinabla} and identity~\eqref{identity:stab-lifting}.

The operator~$\Gfrakh$ introduced in~\eqref{generalised-gradient}
allows for an equivalent rewriting of the VEM.

\begin{thm}[Rewriting of the VEM]
\label{thm:vem-rewriting}
Let $\E$ in~$\taun$.
For all $\phi_h$ and~$\vh$ in~$\VhE$, we have
\[
\ahE(\phi_h,v_h) = (\GfrakhE(\phi_h),\nabla v_h)_\E .
\]
Similarly, if $\phi_h$ and~$v_h$ belong to~$V_h$, then
\[
\ah(\phi_h,v_h) = (\Gfrakh(\phi_h),\nabla v_h)_\Omega .
\]
\end{thm}
\begin{proof}
Let $\phi_h$ and~$\vh$ be in~$\VhE$.
We split the local discrete bilinear form~$\ahE(\cdot,\cdot)$
into three terms:
\begin{equation*}
\ahE(\phi_h,v_h)
=
(\nabla(\Pinablap \phi_h),\nabla(\Pinablap v_h))_\E
+
\SE( (I-\Pinablap) \phi_h, v_h )
-
\SE( (I-\Pinablap) \phi_h, \Pinablap v_h ).
\end{equation*}
Since $\Pinablap$ is an orthogonal projection
and using~\eqref{Pinabla} with $\qp = \Pinablap\phi_\h$,
we have
\begin{equation*}
(\nabla(\Pinablap \phi_h),\nabla(\Pinablap v_h))_\E
=
(\nabla(\Pinablap \phi_h),\nabla v_h)_\E
\end{equation*}
and, by the definition of $\Scalh := \Scalh^K(\phi_h)$
in~\eqref{definition:Scalh},
\begin{equation*}
\SE( (I-\Pinablap) \phi_h, \Pinablap v_h )
=
(\nabla \Scalh,\nabla(\Pinablap v_h))_\E
=
(\nabla \Scalh,\nabla v_h)_\E.
\end{equation*}
This leads us to
\begin{equation*}
\ahE(\phi_h,v_h)
=
(\nabla(\Pinablap \phi_h-\Scalh),\nabla v_h)_\E
+
\SE( (I-\Pinablap) \phi_h, v_h ).
\end{equation*}
For the remaining term, we introduce $\thetaboldh := \thetaboldh^K(\phi_h)$
and invoke Lemma \ref{lem:stab_lift}, giving
\begin{equation} \label{why-RT?}
\begin{split}
\SE( (I-\Pinablap) \phi_h, v_h )
& = -(r_h^K(\phi_h),v_h)_K + (\mu_h^K(\phi_h),v_h)_{\partial K} \\
& = -(\div \thetaboldh,v_h)_K + (\thetaboldh \cdot \bn,v_h)_{\partial K}
= (\thetaboldh,\nabla v_h)_K,
\end{split}
\end{equation}
thereby concluding the proof
by picking~$\thetaboldh$ as that minimizing the $L^2$ norm
over all possible Raviart-Thomas functions
satisfying the constraints in~\eqref{why-RT?}.
\end{proof}

The use of Raviart-Thomas elements over the subtriangulation
is crucial in the construction of~$\thetaboldh$, see~\eqref{why-RT?},
notably in incorporating information
on the stabilisation in the generalised gradient.

\section{Technical results} \label{section:technical-results}
The aim of this section is to discuss some technical results,
which are useful in the a posteriori error analysis in Sections~\ref{section:eq:prager_synge},
\ref{section:flux}, and~\ref{section:potential} below.
In Section~\ref{subsection:partition-of-unity}, we introduce a virtual partition of unity and analyse its properties.
In Section~\ref{subsection:discrete-stable-minimisation},
we investigate stability properties of certain minimisation problems.

\subsection{Virtual element partition of unity} \label{subsection:partition-of-unity}
We introduce a virtual element partition of unity over the mesh~$\taun$,
which we shall employ below in the localisation
of the computation of the error estimator.
For each vertex~$\nu$ of~$\Vcaln$, we define~$\phinu$
as the only function in~$\Vh$
that is harmonic in each element,
piecewise linear on the skeleton of~$\taun$,
equal to~$1$ at the vertex~$\nu$,
and equal to~$0$ at all other vertices.
This set of functions is a partition of unity;
see, e.g., \cite[Section~2.1]{Perugia-Pietra-Russo:2016}.

For $\nu$ in~$\Vcaln$, we denote the set of $\E$
in~$\taun$ sharing the vertex $\nu$ by $\taunnu$.
Then, 
\begin{equation} \label{vertex-patch}
\omeganu := \operatorname{supp} \phinu
\end{equation}
corresponds to the set covered by the elements in $\taunnu$.
We denote the set of facets sharing the vertex $\nu$ by~$\Ecalnu \subset \Ecaln$
and, if~$\nu$ is a boundary vertex,
its subset of (two) facets lying on the boundary $\partial \Omega$ by~$\EcalnuB$.

\begin{subequations}
The space $\Htildeonenu \subset H^1(\omeganu)$
defined for~$\nu$ in~$\VcalnI$ as
\begin{equation} \label{eq:defn_Htildeonenu}
\Htildeonenu
:= \left \{ v \in H^1(\omeganu) \; | \; (v,1)_{\omeganu} = 0 \right \}
\end{equation}
and for $\nu$ in~$\VcalnB$ as
\begin{equation}
\Htildeonenu
:= 
\left \{ v \in H^1(\omeganu) \; | \; v_{|\e} = 0 \quad  
                \forall e \in \EcalnuB \right \}
\end{equation}
will be useful later on, as well as 
\end{subequations}
the broken counterpart defined by $\widetilde H^1(\taunnu) := H^1(\taunnu)$
and $\widetilde H^1(\taunnu) := \{ v \in H^1(\taunnu) \; | \;
(v,1)_{\omeganu} = 0 \}$
for boundary and interior vertices, respectively.

We introduce the average operator over any facet~$\e$ in~$\Ecaln$:
\begin{equation*}
\Pizez v = \he^{-1}(v,1)_{\e} 
\qquad\qquad
\forall v \in L^1(\e).
\end{equation*}
Similarly,
given~$\nbfE$ the outward unit normal vector to an element~$\E$
and~$\nbfe$ one of the two normal vectors to a facet~$\e$,
we define the jump operator
\[
\jump{\theta}:=
\begin{cases}
\theta_{|\E_1} (\nbf_{\E_1} \cdot \nbfe)
    + \theta_{|\E_2} (\nbf_{\E_2} \cdot \nbfe) 
    & \text{if } \e\in\EcalnI,\; \e \subseteq \partial \E_1 \cap \partial \E_2 \\
\theta_{|\E} (\nbf_{|\E} \cdot \nbfe)
    & \text{if } \e\in\EcalnB,\; \e \subseteq \partial \E\\
\end{cases}
\qquad\quad \forall  \theta \in \widetilde H^1(\taunnu).
\]

\begin{remark} \label{remark:things-simplices-DG}
In the case of simplicial meshes with finite element and discontinuous Galerkin methods~\cite{Ern-Vohralik:2015},
the standard partition of unity given
by hat functions~$\phinu$ satisfies some standard properties.
More precisely, the scaling properties
\begin{equation} \label{PU-FEM-1}
\|\phinu\|_{L^\infty(\Omega)} = 1,
\qquad\qquad
\|\nabla \phinu\|_{L^\infty(K)} \leq C(\gamma) h_K^{-1} \quad \forall K \in \taun ,
\end{equation}
combined with the product rule
and a broken Poincar\'e inequality as in, e.g., \cite{Brenner:2003},
yield the inequality
\begin{equation} \label{PU-FEM-2}
\|\nabla_h(\phinu\theta)\|_{\omeganu}^2
\le
\Ccontnu^2
(\|\nabla_h\theta\|_{\omeganu}^2
+
\sum_{e \in \Ecalnu} \he^{-1}\|\Pizez\llbracket \theta \rrbracket\|_e^2)
\qquad\qquad
\forall \theta \in \widetilde H^1(\taunnu),
\end{equation}
where the constant $\Ccontnu$ only depends
on the regularity parameter~$\gamma$
in Section~\ref{subsection:meshes}.
\end{remark}

We establish corresponding properties as those in Remark~\ref{remark:comparison_DG} for the case of polytopic meshes.
We have the following result generalising~\eqref{PU-FEM-1}.

\begin{lem} \label{lemma:partition-of-unity1}
For every~$\nu$ in~$\Vcaln$,
the partition of unity function~$\phinu$ satisfies
\begin{equation} \label{identity&bound:partition-of-unity}
\Norm{\phinu}{L^\infty(\Omega)} = 1,
\qquad\qquad
\Norm{\nabla\phinu}{L^4(\E)}
\le C(\gamma) \hE^{-\frac12} 
\qquad \forall \E \in \taun.
\end{equation}
\end{lem}
\begin{proof}   
The identity in~\eqref{identity&bound:partition-of-unity}
follows from the maximum principle for harmonic functions and the fact that~$\phinu$ is piecewise linear over the skeleton of the mesh,
equal to~$1$ at the vertex~$\nu$,
and~$0$ at all other vertices.

As for the bound in~\eqref{identity&bound:partition-of-unity},
we first recall the scaled continuous Sobolev embedding \cite{Ern-Guermond:2021}
$H^{\frac12}(\E) \hookrightarrow L^4(\E)$ given by
\[
\Norm{\nabla \phinu}{L^4(\E)}
\leq
C(\gamma) \left(
\hE^{-\frac12} \Norm{\nabla \phinu}{\E}
+ \SemiNorm{\nabla \phinu}{\frac12,\E} \right).
\]
Since~$\phinu$ is the solution to a Laplace problem with piecewise continuous linear polynomial Dirichlet boundary conditions,
we use standard a priori estimates for elliptic problems,
a (lowest order) polynomial inverse inequality on~$\partial\E$,
and the fact that the mesh has no ``small facets'',
and deduce
\[
\Norm{\nabla \phinu}{L^4(\E)}
\le C(\gamma) \hE^{-1} \Norm{\phinu}{\partial\E}
\leq C(\gamma) \hE^{-1} \Norm{\phinu}{L^{\infty}(\partial\E)} \vert \partial\E \vert^{\frac12}
= C(\gamma) \hE^{-1} \vert \partial\E \vert^{\frac12}
\le C(\gamma) \hE^{-\frac12}.
\]
\end{proof}

We also have the following result generalising~\eqref{PU-FEM-2},
which follows from the chain rule, the piecewise Sobolev embeddings
$H^1(T) \hookrightarrow L^4(T)$ over the subtriangulation,
Lemma~\ref{lemma:partition-of-unity1},
and a broken Poincar\'e inequality~\cite{Brenner:2003}
on the subtriangulation.

\begin{cor}
For every~$\nu$ in~$\Vcaln$,
there exists a positive constant~$\Ccontnu$ only depending
on the regularity parameter~$\gamma$
in Section~\ref{subsection:meshes} such that
\begin{equation} \label{eq:Ccontnu}
\|\nabla_h(\phinu\theta)\|_{\omeganu}^2
\leq \Ccontnu^2
(\|\nabla_h\theta\|_{\omeganu}^2
+ \sum_{e \in \Ecalnu} \he^{-1} \| \Pizez \llbracket \theta \rrbracket\|_{e}^2)
\qquad\qquad \forall \theta \in \widetilde H^1(\tautildennu).
\end{equation}
\end{cor}

\subsection{Discrete stable minimisation} \label{subsection:discrete-stable-minimisation}
\begin{subequations} \label{eq_stab}

In the proof of the efficiency of the error estimator,
we shall need two technical bounds:
for each vertex $\nu$ in~$\Vcaln$,
there exists a constant $\Cstabnu$ such that
\begin{equation} \label{eq_stab_H1}
\min_{v_h \in \Pbb_{\p+2}(\tautildennu) \cap \Honenu}
\|\bG_h-\nabla v_h\|_{\omeganu}
\leq \Cstabnu \min_{v \in \Honenu}
\|\bG_h-\nabla v\|_{\omeganu}
\end{equation}
and
\begin{equation} \label{eq_stab_Hd}
\min_{\substack{\btau_h \in \RT_{\p}(\tautildennu) \cap \bH(\div,\omeganu)}
\\
\div \btau_h = r_h}
\|\bG_h+\btau_h\|_{\omeganu}
\leq \Cstabnu \min_{\substack{ \btau \in \bH(\div,\omeganu) \\
\div \btau = r_h}}
\|\bG_h+\btau\|_{\omeganu}
\end{equation}
for all $\bG_h$ in~$\RT_{\p}(\tautildennu)$
and~$r_h$ in~$\Pbb_{\p}(\tautildennu)$.
\end{subequations}

Under Assumption \ref{assumption_faces}, in Appendix~\ref{subappendix:p-indep}
we show that~$\Cstabnu$ depends on the shape-regularity parameter~$\gamma$
but not on the degree of accuracy~$\p$.
In the general case, in Appendix~\ref{subappendix:p-dep},
we show that~\eqref{eq_stab} holds true
with a constant depending
on the regularity parameter~$\gamma$
in Section~\ref{subsection:meshes} and possibly on~$\p$.
We conjecture that the dependence on~$\p$ for general meshes is artificial
and could be removed at the price of a more technical proof.

\section{A posteriori error estimator}\label{section:eq:prager_synge}

The goal of this section is to present the general framework
on which the construction of an
a posteriori estimator estimator is based
and propose an intuitive motivation for its definition.
We describe how the estimator can be efficiently computed in practice,
and state its key reliability and efficiency properties in
Theorem \ref{theorem:reliability-efficiency}.
The actual proof of Theorem~\ref{theorem:reliability-efficiency} is postponed
to Sections~\ref{section:flux} and~\ref{section:potential} below.

The remainder of the section is organised as follows.
We introduce the
sufficient condition we require for the forthcoming
analysis on generalised gradients in Section~\ref{subsection:gg}
and use it to derive a generalised Prager-Synge identity in Section~\ref{subsection:a-generalised-PS-identity}.
We introduce the error estimator and motivate its structure in Section~\ref{subsection:motivation},
discuss its practical computability in Section~\ref{subsection:practical-computation},
and state the main result of the paper, namely Theorem~\ref{theorem:reliability-efficiency},
along with some comments
in Section~\ref{subsection:rel-eff}.

\subsection{Generalised gradient} \label{subsection:gg}
The results of this section hold true
for any generalised gradient satisfying the following assumption.
\begin{assumption}[Generalised gradient]
\label{assumption:galerkin-orthogonality}
$\Gcalh$ in~$\RT_\p(\tautilden)$ is a generalised gradient
for~$\uh$ solution to~\eqref{VEM} if
\begin{equation*}
(\Gcalh,\nabla \phinu)_\Omega = (f,\phinu)_\Omega
\qquad\qquad \forall \nu \in \VcalnI .
\end{equation*}
\end{assumption}
The forthcoming analysis applies for all generalised gradients
satisfying Assumption~\ref{assumption:galerkin-orthogonality}.
$\Gcalh$ in~\eqref{generalised-gradient} is a practical realisation
of a generalised gradient satisfying the Galerkin orthogonality assumption.

\subsection{A generalised Prager-Synge identity} \label{subsection:a-generalised-PS-identity}
The construction of the error estimator 
hinges upon a generalised Prager--Synge identity,
as introduced in~\cite{Ern-Vohralik:2015}
for discontinuous Galerkin methods.
For completeness, we adapt the proof of~\cite[Theorem 3.3]{Ern-Vohralik:2015}
to the current setting.

\begin{thm}[Prager--Synge]
For all $\Gcalh \in \bL^2(\Omega)$
and given~$u$ the solution to~\eqref{weak-formulation},
we have
\begin{equation}
\label{eq:prager_synge}
\|\nabla u-\Gcalh\|_\Omega^2
=
\min_{s \in H^1_0(\Omega)} \|\Gcalh-\nabla s\|_\Omega^2
+
\sup_{\substack{v \in H^1_0(\Omega) \\ \|\nabla v\|_\Omega = 1}}
\left \{ (f,v)_\Omega-(\Gcalh,\nabla v)_\Omega \right \}^2.
\end{equation}
\end{thm}

\begin{proof}
Let~$s^\star$ be the orthogonal projection of $\Gcalh$
onto $H^1_0(\Omega)$, i.e.,
\begin{equation}
\label{definition:s}
(\nabla s^\star,\nabla v) = (\Gcalh, \nabla v)
\qquad\qquad\qquad\qquad \forall v \in H^1_0(\Omega).
\end{equation}
Pythagoras' theorem gives
\begin{equation*}
\|\nabla u - \Gcalh\|_\Omega^2
= \|\Gcalh - \nabla s^\star\|_\Omega^2
+ \|\nabla (u-s^\star)\|_\Omega^2
= \min_{s \in H^1_0(\Omega)}
\|\Gcalh - \nabla s\|_\Omega^2
+ \|\nabla (u-s^\star)\|_\Omega^2.
\end{equation*}
Using that~$u$ is the solution to~\eqref{weak-formulation}
and definition~\eqref{definition:s},
the conclusion is a consequence of the following identities:
\[
\begin{split}
\|\nabla (u-s^\star)\|_\Omega
& = \sup_{\substack{v \in H^1_0(\Omega) \\ \|\nabla v\|_\Omega = 1}}
(\nabla(u-s^\star), \nabla v)_\Omega \\
& = \sup_{\substack{v \in H^1_0(\Omega) \\ \|\nabla v\|_\Omega = 1}}
\left \{(f,v)_\Omega- (\nabla s^\star, \nabla v)_\Omega\right \}
= \sup_{\substack{v \in H^1_0(\Omega) \\ \|\nabla v\|_\Omega = 1}}
\left \{ (f,v)_\Omega - (\Gcalh, \nabla v)_\Omega \right \}.
\end{split}
\]
\end{proof}

\subsection{Motivation and definition of the error estimator} \label{subsection:motivation}
We can employ duality in optimisation
to control the last term
on the right-hand side of~\eqref{eq:prager_synge} as
\begin{equation*}
\sup_{\substack{v \in H^1_0(\Omega) \\ \|\nabla v\|_\Omega = 1}}
\{(f,v)_{\Omega} - (\Gcalh, \nabla v)_{\Omega}\}
\leq
\min_{\substack{\sigmabold \in \bH(\div,\Omega) \\ \div \sigmabold = f}}
\|\Gcalh+\sigmabold\|_{\Omega}.
\end{equation*}
Therefore, identity~\eqref{eq:prager_synge} states that the error
is bounded by the sum of two contributions
measuring the lack of $H^1_0(\Omega)$ and $\bH(\div,\Omega)$ conformity of~$\Gcalh$, i.e.,
\begin{equation}
\label{nabla:u-Gcalh}
\|\nabla u-\Gcalh\|_\Omega^2
\leq
\min_{s \in H^1_0(\Omega)}\|\Gcalh - \nabla s\|_\Omega^2
+
\min_{\substack{\sigmabold \in \bH(\div,\Omega) \\ \div \sigmabold = f}}
\|\Gcalh+\sigmabold\|_\Omega^2.
\end{equation}
In the construction of the error estimator,
we measure these lacks of conformity
using local computations on vertex patches
as in~\eqref{vertex-patch}.
In other words, we expect that
\begin{equation}
\label{eq:broken_prager_synge}
\|\nabla u-\Gcalh\|_\Omega^2
\simeq
\sum_{\nu \in \Vcaln}
\Big \{
\min_{\snu \in \Htildeonenu}\Norm{\Gcalh - \nabla \snu}{\omeganu}^2
+
\min_{\substack{\sigmaboldnu \in \bH(\div,\omeganu) \\ \div \sigmaboldnu = f}}
\|\Gcalh+\sigmaboldnu\|_{\omeganu}^2
\Big \}
\end{equation}
in some suitable sense, where $\Htildeonenu$ is defined in \eqref{eq:defn_Htildeonenu}.
This idea of ``broken equilibration'' was already partially explored in~\cite{ChaumontFrelet-Ern-Vohralik:2022}
but only for the second term.

\begin{subequations}
\label{eq:definition_etanu}
The minimisation problems on the right-hand side are local
but still not computable.
Due to the discrete stable minimisations in Section~\ref{subsection:discrete-stable-minimisation},
the proposed error estimator
is based on replacing the minimisation problems
with their natural finite element counterparts, namely,
\begin{equation} \label{etanuPOT}
\etanuPOT
:= \min_{v_h \in \Htildeonenu \cap \Pbb_{\p+2}(\tautildennu)} \Norm{\Gcalh - \nabla v_h}{\omeganu}
\end{equation}
and
\begin{equation} \label{etanuFLUX}
\etanuFLUX
:=
\min_{\substack{\btau_h \in \bH(\div,\omeganu) \cap \RT_\p(\tautildennu) \\ \div \btau_h = f}}
\|\Gcalh+\btau_h\|_{\omeganu}.
\end{equation}
The minimisers, which we shall henceforth denote by~$\snuh$ and~$\sigmaboldnuh$,
are computable as solution to local finite element problems,
so that the proposed estimator is computable.
\end{subequations}

\subsection{Practical computation of the estimator on vertex patches} \label{subsection:practical-computation}
In order to compute~$\etanuPOT$ in~\eqref{etanuPOT},
we solve
\begin{equation}
\label{primal-FEM-on-patches}
\begin{cases}
    \text{find } \snuh \in \Pbb_{\p+2}(\tautildennu) \cap \Htildeonenu \text{ such that}\\
    (\nabla \snuh, \nabla \zh)_{\omeganu} = (\Gcalh, \nabla \zh)_{\omeganu} &
    \forall \zh \in \Pbb_{\p+2}(\tautildennu) \cap \Htildeonenu,
\end{cases}
\end{equation}
where the mean-value may be fixed arbitrarily for internal vertices,
since only $\nabla \snuh$ appears in the definition of the estimator.

In order to compute~$\etanuFLUX$ in~\eqref{etanuFLUX},
we have to compute the discrete minimiser in~\eqref{etanuFLUX},
i.e., the discrete flux solving
\small{\begin{equation} \label{mixed-FEM-on-patches}
\begin{cases}
\text{find } \sigmaboldnuh \in \RT_\p(\tautildennu) \cap \bH(\div,\omeganu) \text{ and } \rnuh \in \Pbb_{\p} (\tautildennu) \text{ such that} \\
(\sigmaboldnuh, \tauboldnuh)_{\omeganu}
        + (\div\tauboldnuh, \rnuh)
        = (-\Gcalh,\tauboldnuh)_{\omeganu} 
        & \forall \tauboldnuh \in \RT_\p(\tautildennu) \cap \bH(\div,\omeganu) \\
(\div\sigmaboldnuh, \qnuh)_{\omeganu}
        = (f,\qnuh)_{\omeganu}
        & \forall \qnuh \in \Pbb_{\p} (\tautildennu).
\end{cases}
\end{equation}}\normalsize{}
The above finite element problems are local on each vertex patch
and can be solved in parallel.

\subsection{Reliability and efficiency} \label{subsection:rel-eff}

For all $\nu$ in~$\Vcaln$, we define the vertex estimator as
\begin{equation*}
\etanu^2
:= \etanuFLUX^2 + \etanuPOT^2 
    + \|\Gcalh-\nabla_h(\Pinablap u_h)\|_{\omeganu}^2
    + \sum_{e \in \Ecalnu} \he^{-1} \|\Pi^0\llbracket \Pinablap u_h \rrbracket\|_e^2
\end{equation*}
and the global error estimator as
\begin{equation} \label{global-error-estimator}
    \eta^2 := \sum_{\nu \in \Vcaln} \etanu^2.
\end{equation}
For a generic union~$D$ of general polygonal elements,
we shall be interested in the error measure
\begin{equation} \label{measure:error}
\ERR^2(D)
:= \|\nabla u - \Gcalh\|_{D}^2 
    + \|\Gcalh - \nablah(\Pinablap u_h)\|_D^2
    + \sum_{\substack{\e \in \Ecaln,\; \e \subset D}}
    \he^{-1} \|\Pi^0 \llbracket \Pinablap u_h \rrbracket\|_{\e}^2.
\end{equation}
In Section~\ref{subsection:nr-comparison-errors} below,
we shall assess the performance of~\eqref{measure:error}
and compare it with a more standard choice:
the two measures behave essentially the same on different
meshes and for different degrees of accuracy.

\begin{thm}
\label{theorem:reliability-efficiency}
Let~$\gamma$ be the shape-regularity constant introduced in Section~\ref{subsection:meshes};
$\Nvert$ the maximal number of vertices per element in the mesh;
$\Ccontnu$ as in~\eqref{eq:Ccontnu};
$\Cstabnu$ as in~\eqref{eq_stab}.
The following reliability
and local efficiency estimates hold true:
\begin{equation} \label{reliability-efficiency}
\begin{split}
& \ERR^2(\Omega)
    \le 4 \Nvert \sum_{\nu} \Ccontnu^2 \etanu^2 
    \le 4 \Nvert \max_{\nu\in \Nucalh} \Ccontnu^2 \sum_{\nu} \etanu^2 
    =: \crel \eta^2, \\
& \etanu^{2}
    \le 2 \Cstabnu^2 \ERR^2(\omeganu)
    =: \ceff \ERR^{2}(\omeganu).
\end{split}
\end{equation}
The constants~$\crel$ and~$\ceff$ are positive,
independent of the choice of the stabilisation
and the mesh size,
but dependent on the regularity parameter~$\gamma$
in Section~\ref{subsection:meshes}.
The constant~$\ceff$ possibly depends also on the degree of accuracy~$\p$
through~$\Cstabnu$.
If two distinct mesh elements only share one facet
as stated in Assumption~\ref{assumption_faces},
then $\ceff$ is independent of~$\p$.
\end{thm}

We conjecture that~$\ceff$ in~\eqref{reliability-efficiency}
is independent of~$\p$ also for meshes,
which do not satisfy Assumption~\ref{assumption_faces}.
\medskip

The computation of the error estimator involves
the computation of the following quantities:
\begin{itemize}
    \item the generalised gradient~$\Gcalh$ in~\eqref{generalised-gradient},
    which can be seen as an elementwise post-processing of the VE solution;
    \item the solution to primal finite element problems on vertex patches
    as in~\eqref{primal-FEM-on-patches} for the computation of~$\etanuPOT$;
    \item the solution to mixed finite element problems on vertex patches
    as in~\eqref{mixed-FEM-on-patches}
    for the computation of~$\etanuFLUX$.
\end{itemize}
The computation of the error estimator is  local and parallelizable,
whence its cost is comparable to that of a residual error estimator.
Furthermore, on meshes satisfying Assumption~\ref{assumption_faces},
it is $\p$-robust and the constants in the upper and lower bound are independent of the choice of the stabilisation.

\begin{remark}[Element-based estimator]
The estimator $\etanu$ is vertex-based. To obtain an element-based estimator, we can set
\begin{equation*}
\etaE^2 := \sum_{\nu \in \VcalE} \etanu^2
\qquad\qquad \forall \E \in \taun.
\end{equation*}
The estimates in \eqref{reliability-efficiency} also hold true for~$\etaE$
if the summation happens on the elements in the upper bound,
and the vertex patch~$\omeganu$
is replaced by the following element patch in the lower bound:
\begin{equation*}
\omegaE = \bigcup_{\nu \in \VcalE} \omeganu.
\end{equation*}
\end{remark}

\begin{remark}[Comparison with discontinuous Galerkin methods] 
\label{remark:comparison_DG}
As detailed in Lemmas~\ref{lem:upper_flux} and~\ref{lem:upper_pot}
below,
the estimator relies on the ``broken'' Prager--Synge estimation~\eqref{eq:broken_prager_synge},
instead of the usual Prager--Synge identity in~\eqref{nabla:u-Gcalh}.

For discontinuous Galerkin schemes,
the authors of~\cite{Ern-Vohralik:2015} constructed two fields
$\sigmabold_h$ in $\bH(\div,\Omega)$ with $\div \sigmabold_h = f$,
and $s_h$ in $H^1_0(\Omega)$ to insert in~\eqref{nabla:u-Gcalh}.
Similar to ours, the construction in~\cite{Ern-Vohralik:2015}
for~$\sigmabold_h$ and~$s_h$ relies on vertex-patch finite element problems,
with the crucial difference that these problems
involve hat functions of the mesh as a partition of unity.
Specifically, for each vertex $\nu$ in~$\Vcaln$,
the local finite element problems
\begin{subequations}
\label{eq:local_contribution_hat}
\begin{equation}
\label{eq:local_contribution_hat_H1}
\snuh
:= \arg \min_{\substack{v_h \in \Pbb_{\p+1}(\taunnu) \cap H^1_0(\omeganu)}}
\|\nabla_h(\phinu u_h)-\nabla v_h\|_{\omeganu}
\end{equation}
and
\begin{equation}
\sigmaboldnu_h
:= \arg \min_{\substack{
\btau_h \in \RT_\p(\taunnu) \cap \bH_0(\div,\omeganu) \\
\div \btau_h = \phinu f - \nabla \phinu \cdot \Gfrakh(u_h)}}
\|\phinu \Gfrakh(u_h)+\btau_h\|_{\omeganu}
\end{equation}
\end{subequations}
are solved with suitable modifications at boundary vertices.

These local contributions are summed to obtain
the globally conforming fields~$s_h$ and~$\sigmabold_h$.
For all $\E$ in~$\taun$,
the lower bound takes the form
\[
\begin{split}
& \|\Gfrakh(u_h)+\sigmabold_h\|_{K}^2+\|\Gfrakh(u_h)-\nabla s_h\|_{K}^2 \\
& \le \Nvert \CstabE^2 \CcontE^2
\Big( \|\nabla u-\Gfrakh(u_h)\|_{\omegaE}^2
+ \|\nabla_h u_h-\Gfrakh(u_h)\|_{\omegaE}^2
+ \sum_{\substack{e \in \Ecaln \\ e \subset \overline{\omegaE}}} \he^{-1}\|\Pi^0\llbracket u_h \rrbracket\|_e^2 \Big),
\end{split}
\]
where
\begin{equation*}
\CstabE := \max_{\nu \in \VcalE} \Cstabnu,
\qquad\qquad\qquad\qquad
\CcontE := \max_{\nu \in \VcalE} \Ccontnu.
\end{equation*}
A key asset of this construction is that the upper-bound in~\eqref{eq:prager_synge} is constant-free,
leading to guaranteed error bounds.
However, the presence (and the computability) of the partition of unity in~\eqref{eq:local_contribution_hat}
is essential to obtain globally equilibrated fields~$s_h$ and~$\sigmabold_h$.
As a result, this approach is fundamentally limited to simplicial
and tensor-product meshes,
where such closed-form partition of unity functions are available.

In contrast, the broken approach we propose here enables to compute the estimator
even if the partition of unity is not explicitly available,
at the price of displacing the constants~$\Nvert$
and~$\Ccontnu$ from the lower bound to the upper bound;
see Theorem~\ref{theorem:reliability-efficiency}, 
and Lemmas~\ref{lem:upper_flux} and~\ref{lem:upper_pot} below.
For this reason,
the constants and terms in the upper and lower bounds here
are a generalisation of those appearing in the
Galerkin schemes in~\cite{Ern-Vohralik:2015},
and therefore they are as good as the state-of-the-art.
\end{remark}

\section{The flux-type term}\label{section:flux}
We prove upper and (local) lower bounds for the second term on the right-hand side
of~\eqref{eq:prager_synge}.

\begin{lem}[Upper bound for the flux-type term]
\label{lem:upper_flux}
Assume that $\Gcalh$ satisfies Assumption \ref{assumption:galerkin-orthogonality}. Then,
the following upper bound holds true:
\begin{equation}
\label{eq:upper_bound_Hd}
\sup_{\substack{v \in H^1_0(\Omega) \\ \|\nabla v\|_{\Omega} = 1}}
\left \{
(f,v)_\Omega - (\Gcalh,\nabla v)_{\Omega}
\right \}^2
\leq
\Nvert \sum_{\nu \in \Vcaln} \Ccontnu^2 \etanuFLUX^2,
\end{equation}
where the constant~$\Ccontnu$ is that appearing
in~\eqref{eq:Ccontnu} and~$\Nvert$ denotes the maximal number
of vertices of the elements of~$\taun$.
\end{lem}
\begin{proof}
We introduce the linear functional~$\Rcal \in (H^1_0(\Omega))'$
given by
\[
\langle \Rcal, v \rangle
:=
(f,v)_{\Omega} - (\Gcalh,\nabla v)_{\Omega} 
\qquad\qquad\qquad\qquad
\forall v \in H^1_0(\Omega)
\]
and its dual norm
\[
\Norm{\Rcal}{*}
:=
\sup_{\substack{v \in H^1_0(\Omega) \\ \|\nabla v\|_\Omega = 1}} \langle \Rcal,v \rangle.
\]
By Assumption~\ref{assumption:galerkin-orthogonality}, we have
\begin{equation}
\label{Rcal-zero-on-Vh}
\langle \Rcal, \phinu \rangle = 0
\qquad\qquad\qquad\qquad
\forall \nu \in \VcalnI .
\end{equation}
Fix $v$ in~$H^1_0(\Omega)$,
and~$\zeta_\nu = 0$ if $\nu \in \VcalnB$ and
$\zeta_\nu = (v,1)_{\omeganu}/|\omeganu|$ if $\nu \in \VcalnI$.
Since the set $\{ \phinu \}_{\nu \in \Vcaln}$
is a partition of unity, using~\eqref{Rcal-zero-on-Vh},
we have
\[
\langle \Rcal, v \rangle
=
\sum_{\nu \in \Vcaln} \langle \Rcal, \phinu v \rangle
=
\sum_{\nu \in \Vcaln} \langle \Rcal, \phinu (v-\zeta_\nu) \rangle .
\]
Introducing the local dual norms
\begin{equation}
\label{dual-norm-nu-Rcal}
\Norm{\Rcal}{\nu,*}
:= \sup_{\substack{v \in H^1_0(\omeganu) \\ \|\nabla v\|_{\omeganu} = 1}}
\langle \Rcal, v \rangle
\end{equation}
and using \eqref{eq:Ccontnu} and the properties of~$\zeta_\nu$,
we arrive at
\[
\langle \Rcal, v \rangle
\leq
\sum_{\nu \in \Vcaln}
\Norm{\Rcal}{\nu,*}\|\nabla(\phinu (v-\zeta^\nu)\|_{\omeganu}
\leq
\sum_{\nu \in \Vcaln}
\Ccontnu \Norm{\Rcal}{\nu,*}\|\nabla v\|_{\omeganu}.
\]
Further using Cauchy-Schwarz' inequality leads to
\begin{equation*}
|\langle \Rcal, v \rangle|^2
\leq \Big( \sum_{\nu \in \Vcaln} \Ccontnu^2 \Norm{\Rcal}{\nu,*}^2 \Big)
\Big( \sum_{\nu \in \Vcaln} \|\nabla v\|_{\omeganu}^2 \Big)
\leq \Nvert \Big( \sum_{\nu \in \Vcaln} \Ccontnu^2 \Norm{\Rcal}{\nu,*}^2 \Big)
\|\nabla v\|_{\Omega}^2 .
\end{equation*}
Since~$v$ is arbitrary, we conclude that
\begin{equation}
\label{splitting:R}
\Norm{\Rcal}{*}^2
\leq
\Nvert
\sum_{\nu \in \Vcaln}
\Ccontnu^2 \Norm{\Rcal}{\nu,*}^2.
\end{equation}
So far, we proved an upper bound for the second term on the right-hand side of~\eqref{eq:prager_synge},
i.e., we split it into the sum of local contributions.
Next, we bound such local
contributions
on the right-hand side of~\eqref{splitting:R}
by means of a quantity that is
computable via the degrees of freedom and can be obtained by means of local computations on
vertex patches.

By Riesz' representation theorem,
there exists~$\rnu$ in~$H^1_0(\omeganu)$ such that
\begin{equation} \label{Riesz-isometry}
(\nabla \rnu,\nabla v)_{\omeganu}
= \langle \Rcal , v \rangle
\qquad \forall v \in H^1_0(\omeganu),
\qquad\qquad 
\Norm{\Rcal}{\nu,*} = \Norm{\nabla\rnu}{\omeganu}.
\end{equation}
The definition of~$\Rcal$ allows us to write
\[
(\Gcalh + \nabla\rnu, \nabla v)_{\omeganu}
= (f,v)_{\omeganu}
\qquad\qquad\qquad\qquad
\forall v \in H^1_0(\omeganu).
\]
It follows that the function $\sigmaboldnu := -(\Gcalh+\nabla \rnu)$ belongs to $\bH(\div,\omeganu)$ with $\div \sigmaboldnu = f$.
As a result, $\sigmaboldnu$ and~$\rnu$
are the solutions to the following mixed problem:
\begin{equation} \label{global-mixed}
\begin{cases}
    \text{find } (\sigmaboldnu,\rnu) \in \bH(\div,\omeganu) \times L^2(\omeganu) \text{ such that} \\
    (\sigmaboldnu, \tauboldnu)_{\omeganu}
            + (\div\tauboldnu, \rnu)
            = (-\Gcalh,\tauboldnu)_{\omeganu} 
            & \forall \tauboldnu \in \bH(\div,\omeganu) \\
    (\div\sigmaboldnu, \qnu)_{\omeganu}
            = (f,\qnu)_{\omeganu}
            & \forall \qnu \in L^2(\omeganu).
\end{cases}
\end{equation}
Due to~\eqref{dual-norm-nu-Rcal}, \eqref{Riesz-isometry}, and~\eqref{global-mixed},
since~$\sigmaboldnu$ solves the above mixed problem, we write
\[
\Norm{\Rcal}{\nu,*}
= \Norm{\sigmaboldnu + \Gcalh}{\omeganu}
= \min_{\tauboldnu \in \bH(\div,\omeganu),\ 
        \div \tauboldnu =f}
        \Norm{\tauboldnu + \Gcalh}{\omeganu}.
\]
Recall that~$\tautildennu$ is the subtriangulation
of~$\omeganu$
obtained by merging the subtriangulations~$\tautildeE$
of all elements~$\E$ contained in~$\omeganu$.
For future convenience, we also define~$\taunnu$
as the set of elements~$\E$ in~$\taun$ that are contained in~$\omeganu$.
Recalling the definition in \eqref{eq:definition_etanu},
we arrive at
\[
\Norm{\Rcal}{\nu,*}
\le \min_{\tauboldnuh \in \RT_\p(\tautildennu),\ 
        \div \tauboldnuh =f}
        \Norm{\tauboldnuh + \Gcalh}{\omeganu}
=: \etanuFLUX.
\]
The bound in~\eqref{eq:upper_bound_Hd} follows from \eqref{splitting:R}.
\end{proof}

Next, we show the (local) lower bound for the flux-type term.

\begin{lem}[Lower bound for the flux-type term]
Let~$\Cstabnu$ be the constant in~\eqref{eq_stab_Hd}.
For all $\nu$ in~$\Vcaln$, we have
\begin{equation*}
\etanuFLUX \leq \Cstabnu \|\nabla u-\Gcalh\|_{\omeganu}.
\end{equation*}
\end{lem}
\begin{proof}
The assertion is a consequence of the results stated in~\eqref{eq_stab_Hd}
and established in Appendix~\ref{appendix:discrete-min}.
Indeed, since $\Gcalh$ and~$-\nabla u$
belong to~$\RT_{\p}(\tautildennu)$ and $\bH(\div,\omeganu)$
with $\div (-\nabla u) = f$, respectively, we have
\small{\[
\begin{split}
\etanuFLUX
& := \min_{\substack{\sigmaboldnuh \in \RT_{\p}(\tautildennu) \cap \bH(\div,\omeganu) \\
\div \sigmaboldnuh = f}}
    \|\Gcalh+\sigmaboldnuh\|_{\omeganu} 
 \leq \Cstabnu \min_{\substack{
                    \sigmaboldnu \in \bH(\div,\omeganu) \\
                    \div \sigmaboldnu = f }}
    \|\Gcalh+\sigmaboldnu\|_{\omeganu}
\leq \Cstabnu \|\Gcalh-\nabla u\|_{\omeganu}.
\end{split}
\]}\normalsize{}
\end{proof}

\section{The potential-type term} \label{section:potential}
We prove upper and (local) lower bounds
for the first term on the right-hand side of~\eqref{eq:prager_synge}.

\begin{lem}[Upper bound for the potential-type term] \label{lem:upper_pot}
For all $\Gcalh \in \bL^2(\Omega)$, the following upper bound holds true
\[
\begin{split}
& \min_{s \in H^1_0(\Omega)}
\|\Gcalh-\nabla s\|_\Omega^2 \\
& \leq 4\Nvert \sum_{\nu \in \Vcaln}
\Ccontnu^2 \Big(
\etanuPOT^2 + \|\Gcalh-\nablah \Pinablap\uh\|_{\omeganu}^2 +
\sum_{e \in \Ecalnu}
\he^{-1}\|\Pi^0\llbracket \Pinablap\uh\rrbracket\|_e^2 \Big),
\end{split}
\]
where the constant~$\Ccontnu$ appears in~\eqref{eq:Ccontnu}
and~$\Nvert$ denotes the maximal number of vertices of the elements of~$\taun$.
\end{lem}

\begin{proof}
Recall that $\snuh$ are the minimisers
in the definition of~$\etanuPOT$ in \eqref{eq:definition_etanu}.
For interior vertices, the mean value of the~$\snuh$
can be freely chosen;
henceforth, we assume that $(\snuh,1)_{\omeganu} = (\Pinablap \uh,1)_{\omeganu}$
whenever $\nu$ belongs to~$\VcalnI$.

We introduce
\begin{equation*}
s_h := \sum_{\nu \in \Vcaln} \phinu \snuh \in H^1_0(\Omega).
\end{equation*}
The triangle inequality entails
\begin{equation*}
\min_{s \in H^1_0(\Omega)} \|\Gcalh-\nabla s\|_\Omega^2
\leq \|\Gcalh-\nabla s_h\|_\Omega^2
\leq 2 \|\Gcalh-\nabla_h \Pinablap\uh\|_\Omega^2
  + 2 \|\nabla_h(\Pinablap\uh-s_h)\|_\Omega^2.
\end{equation*}
As for the first term, we have
\begin{equation*}
\|\Gcalh-\nablah \Pinablap\uh\|_\Omega^2
\leq
\sum_{\nu \in \Vcaln} \|\Gcalh-\nablah \Pinablap\uh\|_{\omeganu}^2 .
\end{equation*}
As for the second term,
we plug in the partition of unity $\{\phinu\}_{\nu \in \Vcaln}$:
\begin{equation*}
\|\nabla_h(\Pinablap\uh- s_h)\|_\Omega^2
=
\Norm{\sum_{\nu \in \Vcaln} \nabla_h(\phinu(\Pinablap\uh-\snuh))}{\Omega}^2
\leq
\Nvert
\sum_{\nu \in \Vcaln}
\left \|
\nabla_h(\phinu(\Pinablap\uh-\snuh))
\right \|_{\omeganu}^2.
\end{equation*}
Next, we employ~\eqref{eq:Ccontnu}
and the fact that~$\snuh$ belongs to~$H^1(\omeganu)$
with $(\snuh,1)_{\omeganu} = (\Pinablap\uh,1)_{\omeganu}$:
\begin{equation*}
\|\nabla_h(\phinu(\Pinablap\uh-\snuh))\|_{\omeganu}^2
\leq \Ccontnu^2
\Big( \|\nabla_h(\Pinablap\uh-\snuh))\|_{\omeganu}^2
+ \sum_{e \in \Ecalnu} \he^{-1} \|\Pi^0 \llbracket \Pinablap\uh \rrbracket\|_e^2 \Big) . 
\end{equation*}
The assertion follows from the triangle inequality:
\begin{equation*}
\|\nabla_h(\Pinablap\uh-\snuh))\|_{\omeganu}^2
\leq
2\|\Gcalh-\nabla_h(\Pinablap \uh)\|_{\omeganu}^2 + 2\|\Gcalh-\nabla\snuh\|_{\omeganu}^2.
\end{equation*}
\end{proof}

Next, we show the (local) lower bound for the potential-type term.
\begin{lem}[Lower bound for the potential-type term]
Let~$\Cstabnu$ be the constant in~\eqref{eq_stab_H1}.
For all $\nu$ in~$\Vcaln$, we have
\begin{equation*}
\etanuPOT \leq \Cstabnu \|\nabla u-\Gcalh\|_{\omeganu}.
\end{equation*}
\end{lem}

\begin{proof}
The assertion is a consequence of the results stated in~\eqref{eq_stab_H1}
and established in Appendix~\ref{appendix:discrete-min}.
Indeed, since~$\Gcalh$ and~$u$ belong to~$\Pbb_{\p+1}(\tautildennu)$
and~$\Honenu$, respectively,
we have
\[
\begin{split}
\etanuPOT
& := \min_{\snuh \in \Pbb_{\p+2}(\tautildennu) \cap \Honenu}
\|\Gcalh-\nabla \snuh\|_{\omeganu} \\
& \leq \Cstabnu \min_{\snu \in \Honenu}
\|\Gcalh-\nabla \snu\|_{\omeganu}
\leq \Cstabnu \|\Gcalh-\nabla u\|_{\omeganu}.
\end{split}
\]
\end{proof}

\section{Numerical results} \label{section:numerical-results}
We present some numerical experiments.
After introducing two test cases,
in Section~\ref{subsection:nr-comparison-errors}
we compare
the accuracy of the generalised gradient $\Gfrakh(\uh)$
as opposed to the standard VE choice~$\nablah(\Pinablap\uh)$.
Next, we assess the performance of the adaptive scheme
led by the error estimator in~\eqref{global-error-estimator}.
After discussing the general structure of the adaptive algorithm in Section~\ref{subsection:adaptive-algorithm},
we show the performance of the $\h$- and $\p$-versions adaptive scheme
in Sections~\ref{subsection:nr-h-version} and~\ref{subsection:nr-p-version}, respectively.
Throughout, we shall employ the ``projected'' stabilisation~\eqref{projected-stab}.

\paragraph*{Test cases.}
We consider two test cases.
On the square domain~$\Omega_1:= (0,1)^2$, we consider
\[
    u_1(x,y):= \sin(\pi \ x) \sin(\pi \ y).
\]
On the L-shaped domain $\Omega_2 := (-1,1)^2 \setminus [0,1) \times (-1,0]$,
given~$(r,\theta)$ the usual polar coordinates at $(0,0)$,
we consider
\[
    u_2(x,y) = \widetilde u_2(r,\theta)
    := r^{\frac23} \sin \left(\frac23 \theta \right) .
\]

\subsection{An alternative error measure} \label{subsection:nr-comparison-errors}
Here, we are interested in comparing the performance of
three error measures for method~\eqref{VEM}.
Let~$u$ and~$\uh$ be the solutions to~\eqref{weak-formulation} and~\eqref{VEM},
and~$\Gcalh$ be the generalised gradient in~\eqref{generalised-gradient}.
Given~$\Pinablap$ as in~\eqref{Pinabla},
the first quantity we are interested in is
\begin{equation} \label{standard-error}
\Norm{\nabla u - \nablah ( \Pinablap \uh )}{\Omega},
\end{equation}
which is the standard error measure in virtual elements.
The second error measure is that defined
on the left-hand side of~\eqref{nabla:u-Gcalh};
the third one is that in~\eqref{measure:error}.
This is apparent from Tables~\ref{tab:u1p1} and \ref{tab:u1p4},
where we compare the three quantities for the exact solution~$u_1$.
We consider sequences of uniform hexahedral meshes with mesh-sizes $\frac{4}{3}2^{-j}$, $j\geq1$, and take~$\p=1$ and~$4$.

\begin{table}[htb]
  \begin{tabular}{|r|lr|lr|lr|}
    \hline
    $j$ & $\Norm{\nabla u - \nablah (\Pinablap \uh)}{\Omega}$ & order & $\Norm{\nabla u - \Gcalh}{\Omega}$ & order & $\ERR(\Omega)$ & order \\ \hline
    $1$ & $1.1892\times 10^{0}$ & -- & $1.1103\times 10^{0}$ & -- & $1.1529\times 10^{0}$ & -- \\
    $2$ & $6.4270\times 10^{-1}$ & $0.89$ & $5.9106\times 10^{-1}$ & $0.91$ & $6.2828\times 10^{-1}$ & $0.88$ \\
    $3$ & $3.3768\times 10^{-1}$ & $0.93$ & $3.0910\times 10^{-1}$ & $0.94$ & $3.1965\times 10^{-1}$ & $0.97$ \\
    $4$ & $1.7244\times 10^{-1}$ & $0.97$ & $1.5754\times 10^{-1}$ & $0.97$ & $1.6085\times 10^{-1}$ & $0.99$ \\
    $5$ & $8.7032\times 10^{-2}$ & $0.99$ & $7.9415\times 10^{-2}$ & $0.99$ & $8.0678\times 10^{-2}$ & $1.00$ \\
    $6$ & $4.3709\times 10^{-2}$ & $0.99$ & $3.9855\times 10^{-2}$ & $0.99$ & $4.0403\times 10^{-2}$ & $1.00$ \\
    $7$ & $2.1901\times 10^{-2}$ & $1.00$ & $1.9962\times 10^{-2}$ & $1.00$ & $2.0218\times 10^{-2}$ & $1.00$ \\
    $8$ & $1.0962\times 10^{-2}$ & $1.00$ & $9.9896\times 10^{-3}$ & $1.00$ & $1.0113\times 10^{-2}$ & $1.00$ \\
    $9$ & $5.4840\times 10^{-3}$ & $1.00$ & $4.9969\times 10^{-3}$ & $1.00$ & $5.0575\times 10^{-3}$ & $1.00$ \\
    $10$ & $2.7427\times 10^{-3}$ & $1.00$ & $2.4990\times 10^{-3}$ & $1.00$ & $2.5290\times 10^{-3}$ & $1.00$ \\
    \hline
  \end{tabular}
  \caption{Errors for the test case $u_1$ on a sequence of uniform hexahedral meshes with~$p=1$.}
  \label{tab:u1p1}
\end{table}

\begin{table}[htb]
  \begin{tabular}{|r|lr|lr|lr|}
    \hline
    $j$ & $\Norm{\nabla u - \nablah (\Pinablap \uh)}{\Omega}$ & order & $\Norm{\nabla u - \Gcalh}{\Omega}$ & order & $\ERR(\Omega)$ & order \\ \hline
    $1$ & $7.2068\times 10^{-3}$ & -- & $7.1145\times 10^{-3}$ & -- & $7.2193\times 10^{-3}$ & -- \\
    $2$ & $1.2802\times 10^{-3}$ & $2.49$ & $1.3543\times 10^{-3}$ & $2.39$ & $1.4225\times 10^{-3}$ & $2.34$ \\
    $3$ & $7.8260\times 10^{-5}$ & $4.03$ & $9.1314\times 10^{-5}$ & $3.89$ & $1.0082\times 10^{-4}$ & $3.82$ \\
    $4$ & $4.3659\times 10^{-6}$ & $4.16$ & $5.3705\times 10^{-6}$ & $4.09$ & $6.2357\times 10^{-6}$ & $4.02$ \\
    $5$ & $2.5425\times 10^{-7}$ & $4.10$ & $3.2919\times 10^{-7}$ & $4.03$ & $3.9151\times 10^{-7}$ & $3.99$ \\
    $6$ & $1.5360\times 10^{-8}$ & $4.05$ & $2.0383\times 10^{-8}$ & $4.01$ & $2.4519\times 10^{-8}$ & $4.00$ \\
    $7$ & $9.4506\times 10^{-10}$ & $4.02$ & $1.2696\times 10^{-9}$ & $4.00$ & $1.5350\times 10^{-9}$ & $4.00$ \\
    \hline
  \end{tabular}
  \caption{Errors for the test case $u_1$ on
  a sequence of uniform hexahedral meshes with~$p=4$.}
  \label{tab:u1p4}
\end{table}

For this example,
the computation of~$\Gcalh$ can be seen as a cheap and local
post-processing of the discrete solution
leading to a comparable approximation of the exact gradient;
for a more general claim, a theoretical comparison would be necessary.

\subsection{The adaptive algorithm} \label{subsection:adaptive-algorithm}
We consider the usual adaptive algorithm based on the four steps procedure
\[
\textbf{SOLVE}
\qquad \Longrightarrow \qquad
\textbf{ESTIMATE}
\qquad \Longrightarrow \qquad
\textbf{MARK}
\qquad \Longrightarrow \qquad
\textbf{REFINE}.
\]
Some details follow:
\begin{itemize}
    \item while solving the virtual element method
    for a given mesh~$\taun$ and degree of accuracy~$\p$,
    there is no need to compute the generalised gradient~$\Gcalh$,
    as it is a tool for the computation of the error estimator,
    and is computed in a subsequent step as a postprocessing of~$\uh$;
    \item while computing the local error estimator~$\etanu$ on each vertex patch~$\omeganu$,
    it is necessary to compute the generalised gradient~$\Gcalh$,
    the vertex contributions~$\etanuFLUX$ and~$\etanuPOT$,
    and the jumps of~$\Pinablap$;
    \item we use D\"orfler's marking strategy~\cite{Dorfler:1996}
    with bulk parameter $\theta=0.5$;
    \item one of the advantages of employing polytopic meshes in adaptive methods resides in easily handling hanging nodes;
    since we are not interested in coarsening the mesh,
    we only employ initial uniform Cartesian and triangular meshes;
    $\h$-refinements are performed
    by a standard splitting of each quadrilateral and triangular shape into four siblings with halved diameter, respectively.
\end{itemize}

\subsection{The $\h$-version} \label{subsection:nr-h-version}
We are interested in assessing numerically
the upper and lower bounds in Theorem~\ref{theorem:reliability-efficiency}.
To this aim,
given~$\eta$ and~$\ERR(\Omega)$ as in~\eqref{global-error-estimator} and~\eqref{measure:error},
we introduce the effectivity index
\begin{equation} \label{effectivity-index}
    \Ical := \frac{\eta}{\ERR(\Omega)}.
\end{equation}
In Figure~\ref{figure:h-adaptive},
we plot the effectivity index under $\h$-uniform and adaptive mesh refinements
with an initial uniform Cartesian mesh of~$4$ elements
for the test case~$u_1$
and~$12$ elements for the test case~$u_2$,
and check whether that remains constant
for degrees of accuracy~$\p=1$ and~$2$.

\begin{figure}[htb]
\centering
\includegraphics[width=0.49\textwidth]{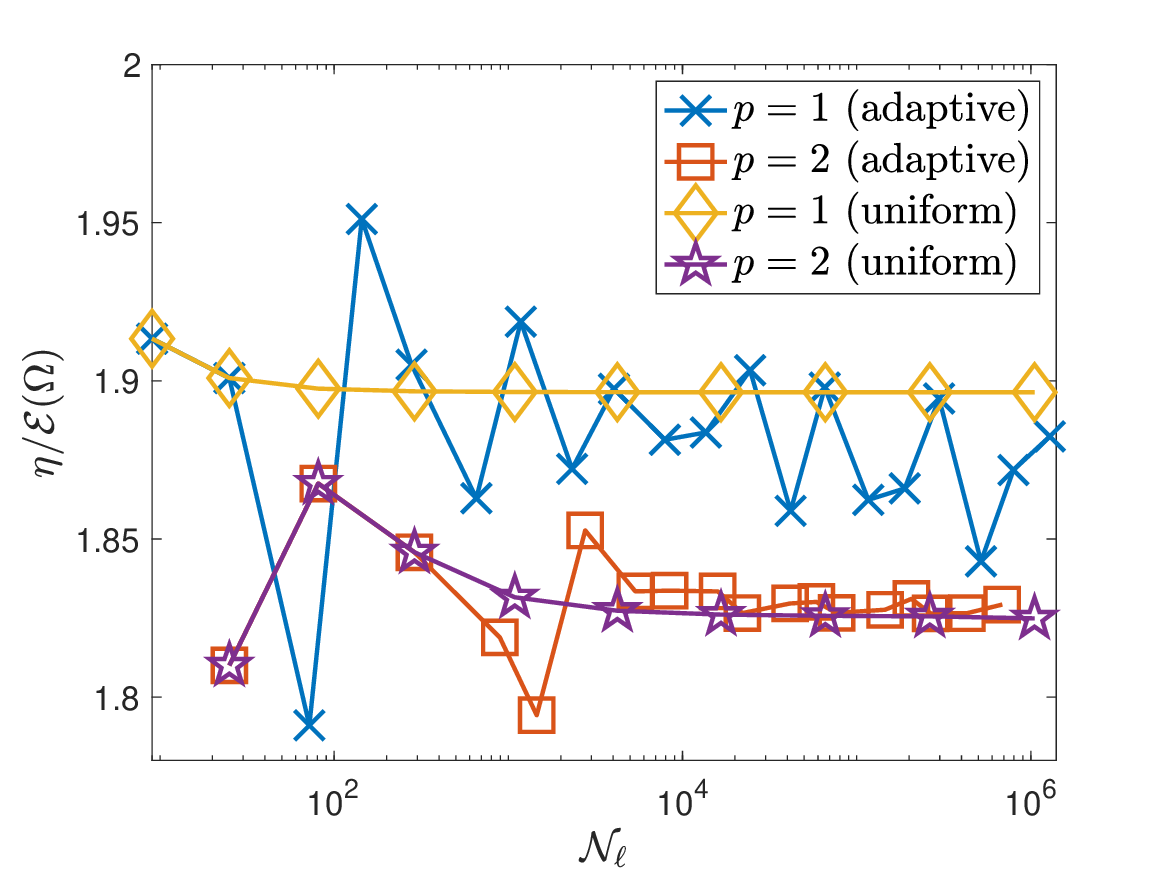}
\includegraphics[width=0.49\textwidth]{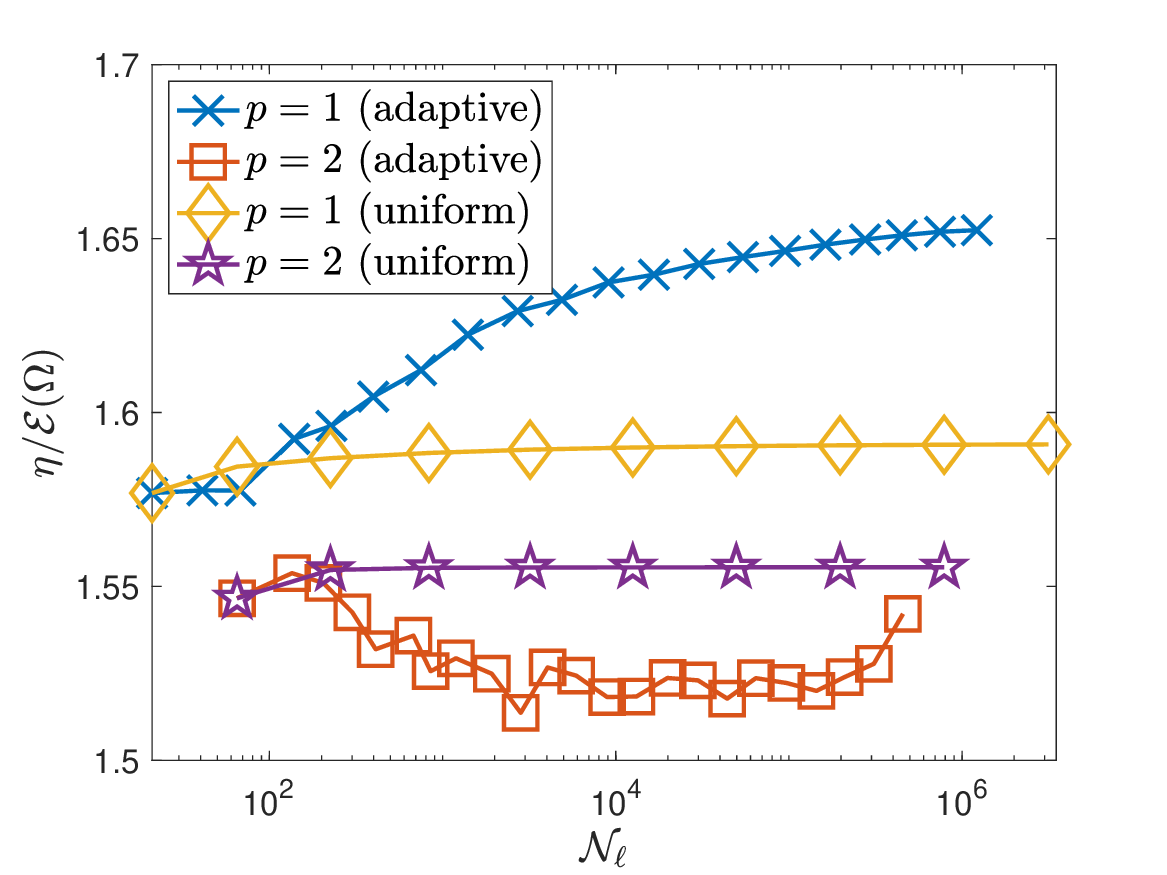}
\caption{Effectivity index~$\Ical$ in~\eqref{effectivity-index}
versus the number of degrees of freedom $\mathcal N_\ell$
for the test cases~$u_1$ (left panel)
and~$u_2$ (right panel)
under uniform and adaptive $\h$-refinements.
We employ an initial uniform Cartesian mesh of 4 or 12 elements, respectively,
and~$\p=1$ and~$2$.}
\label{figure:h-adaptive}
\end{figure}

From Figure~\ref{figure:h-adaptive},
it is apparent that the effectivity index remains in the bounded
interval $(1.5,2)$.

\subsection{The $\p$-version} \label{subsection:nr-p-version}
We check the behaviour of the effectivity index~$\Ical$ under uniform $\p$-refinements, say, up to~$\p=7$,
on a uniform Cartesian, a uniform triangular, and a uniform hexagonal mesh.
We pick the exact solution~$u_2$
and show the results in Figure~\ref{figure:p-uniform}.

\begin{figure}[htb]
\centering
\includegraphics[width=0.50\textwidth]{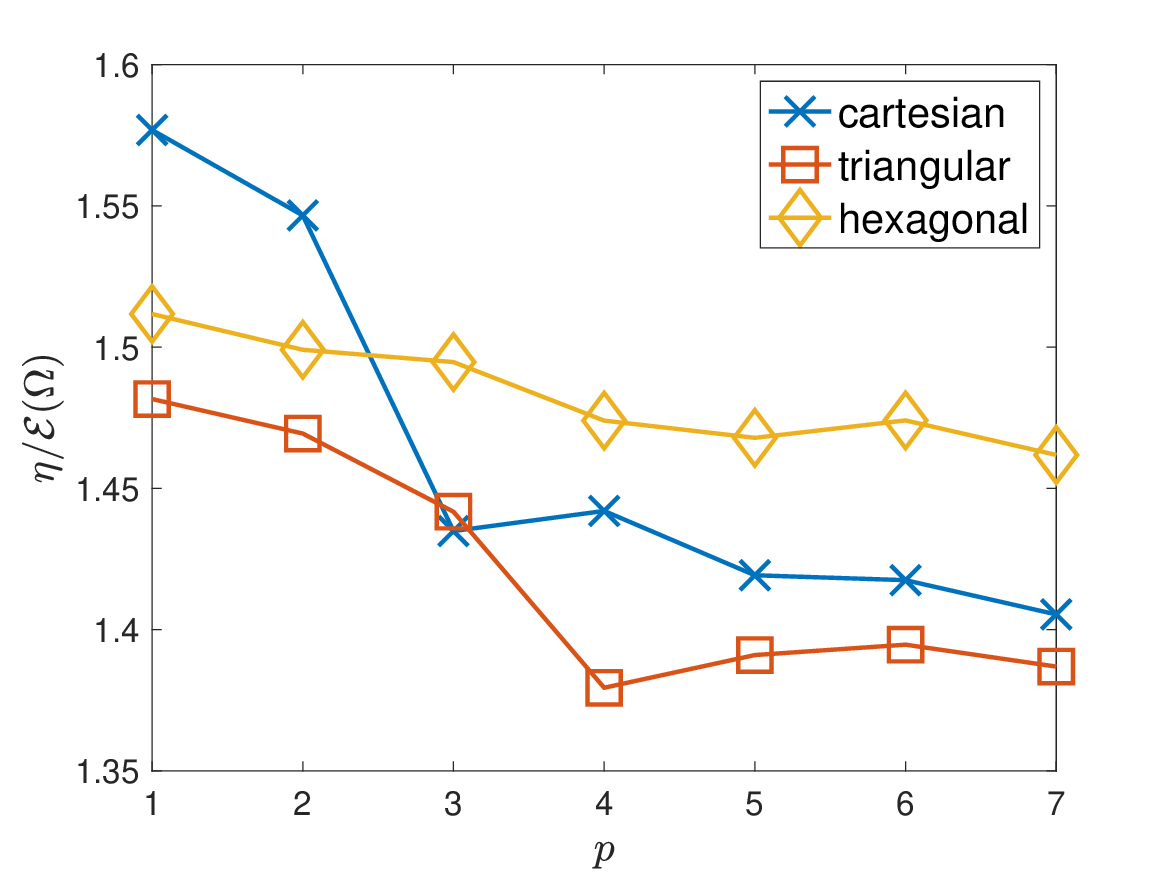}
\caption{Effectivity index~$\Ical$ in~\eqref{effectivity-index}
for the test case~$u_2$ under uniform $\p$-refinements.
}
\label{figure:p-uniform}
\end{figure}

Figure~\ref{figure:p-uniform} illustrates the $\p$-robustness
of the proposed error estimator:
the effectivity index remains in the bounded interval $(1.35,1.6)$.

\section{Conclusions} \label{section:conclusions}

We proposed a generalised gradient formulation for the standard nodal virtual element method in two dimensions.
On the one hand, the employed generalised gradient can be seen as a post-processing of the discrete solution,
which provided us with an alternative computable
approximation of the exact solution.

On the other hand, the generalised gradient formulation allowed us to improve the state-of-the-art of a posteriori error estimates in virtual elements
and other polytopic methods in general,
yet keeping in mind that a modified error measure is considered.
In fact, standard a posteriori error bounds on general meshes for arbitrary order degree of accuracy~$\p$
typically involve constants depending on~$\p$ and the choice of the stabilisations;
the reason for this is to be sought both in the choice of the computable error measure and the error estimator.
By changing these two quantities,
for the first time in the literature on polytopic elements
we ended up with robust a posteriori error estimates.
The proposed error estimator contains terms and constants
that extend those appearing while using other nonconforming methods,
such as the discontinuous Galerkin method,
on standard meshes.

\paragraph*{Acknowledgments.}
LM has been partially funded by MUR (PRIN2022 research grant n. 202292JW3F).
LM was also partially supported by the European Union (ERC Synergy, NEMESIS, project number 101115663).
Views and opinions expressed are however those of the author(s) only and do not necessarily reflect
those of the European Union or the European Research Council Executive Agency. 
LM is member of the Gruppo Nazionale Calcolo Scientifico-Istituto Nazionale di Alta Matematica (GNCS-INdAM).

{\footnotesize
\bibliography{bibliogr}
}
\bibliographystyle{plain}

\appendix

\section{Discrete stable minimisation}
\label{appendix:discrete-min}
Given a vertex~$\nu$ in~$\Vcaln$,
and data $\bG_h$ and $r_h$
in $\RT_\p(\tautildennu)$ and $\Pbb_{\p}(\tautildennu)$,
we investigate the link between the discrete
\begin{equation}
m^{\rm g}_h
:=
\min_{v_h \in \Pbb_{\p+2}(\tautildennu) \cap \Honenu}
\|\bG_h-\nabla v_h\|_{\omeganu},
\qquad
m^{\rm d}_h
:=
\min_{\substack{
\btau_h \in \RT_{\p}(\tautildennu) \cap \bH(\div,\omeganu)
\\
\div \btau_h = r_h
}}
\|\btau_h+\bG_h\|_{\omeganu}
\end{equation}
and continuous minimisation problems
\begin{equation}
m^{\rm g}
:=
\min_{v \in \Honenu} \|\bG_h-\nabla v\|_{\omeganu},
\qquad
m^{\rm d}
:=
\min_{\substack{\btau \in \bH(\div,\omeganu) \\ \div \btau = r_h}}
\|\btau+\bG_h\|_{\omeganu}.
\end{equation}
While it is clear that $m^{\rm g} \leq m^{\rm g}_h$
and $m^{\rm d} \leq m^{\rm d}_h$, the goal of this appendix
is to show the converse inequality up to a constant:
in Section~\ref{subappendix:p-dep}
the inequality is shown for general meshes
with constants possibly depending on~$\p$;
in Section~\ref{subappendix:p-indep},
under more technical assumptions on the mesh,
we prove that the constant does not depend on~$\p$.

\subsection{Bounds for fully general meshes} \label{subappendix:p-dep}
\begin{thm}[Non-$\p$-robust stable minimisation]
We have
\[
m^{\rm g}_h \leq C(\p,\gamma) m^{\rm g},
\qquad\qquad\qquad
m^{\rm d}_h \leq C(\p,\gamma) m^{\rm d},
\]
where the constants are independent of the data $r_h$ and $\bG_h$.
\end{thm}

\begin{proof}
First, we establish the existence of generic constants $C^{\rm g}$ and $C^{\rm d}$
such that $m^{\rm g}_h \leq C^{\rm g} m^{\rm g}$
and $m^{\rm d}_h \leq C^{\rm d} m^{\rm d}$.
We shall investigate the scaling properties in a second step.
We only focus on $C^{\rm g}$;
the estimate involving $C^{\rm d}$ is established similarly.

Given~$s_h$ and~$s$ the discrete and continuous minimisers,
the operators
\[A^{\rm g}_h: \bG_h \in \RT_{\p}(\tautildennu) \to \Pbb_{\p+1}(\tautildennu) \ni \bG_h-\nabla s_h
\]
and
\[A^{\rm g}: \bG_h \in \RT_{\p}(\tautildennu) \to \bL^2(\tautildennu) \ni \bG_h-\nabla s
\]
are linear.

Hence, the following applications
are seminorms on $\RT_{\p}(\tautildennu)$:
\begin{equation*}
\bG_h \to \|A^{\rm g}(\bG_h)\|_{\omeganu},
\qquad\qquad
\bG_h \to \|A^{\rm g}_h(\bG_h)\|_{\omeganu}.
\end{equation*}
Since the space $\RT_{\p}(\tautildennu)$ is finite dimensional,
if we can show that $A^{\rm g}_h(\bG_h) = 0$ whenever $A^{\rm g}(\bG_h) = 0$,
then these seminorms have the same kernel.
Thus, since all seminorms with the same kernel on finite dimensional spaces are equivalent,
we establish that there exists a constant~$C^{\rm g}$ such that
\begin{equation}
\label{tmp_norm_equiv}
\|A^{\rm g}(\bG_h)\|_{\omeganu} \leq C^{\rm g}\|A^{\rm g}_h(\bG_h)\|_{\omeganu}.
\end{equation}
Estimate~\eqref{tmp_norm_equiv} would then imply
that $m^{\rm g} \leq C^{\rm g} m^{\rm g}_h$
for all~$\bG_h$ in $\Pbb_{\p+1}(\tautildennu)$, which concludes the argument.

Therefore, we consider~$\bG_h$ in $\RT_{\p}(\tautildennu)$ such that $A^{\rm g}(\bG_h) = 0$.
We have that~$\bG_h = \nabla s$ in~$\Honenu$ for the continuous minimiser~$s$,
so that $\bG_h$ belongs to~$\nabla(\Pbb_{\p+2}(\tautildennu) \cap \Honenu)$.
It follows that $\bG_h = \nabla s_h$ for the discrete minimiser too,
leading to $A^{\rm g}_h(\bG_h) = 0$.

The constants $C^{\rm g}$ and $C^{\rm d}$ obtained above depend on the space
$\Pbb_{\p+1}(\tautildennu)$, meaning that
they depend on~$\p$ and~$\tautildennu$.
We can show that the dependence on~$\tautildennu$
occurs only through~$\gamma$ by using
standard scaling arguments involving the Piola map and reference patches.
\end{proof}

\subsection{Polynomial-degree-robust bounds} \label{subappendix:p-indep}
Here, under Assumption~\ref{assumption_faces},
we show that the discrete stable minimisation property
holds true with a constant independent of~$\p$.
We start by considering the case of interior vertices.

\begin{lem}[$\p$-robust stable minimisation for interior vertices]
\label{lemma:discrete_stab_interior}
Let Assumption~\ref{assumption_faces} hold true.
Assume that~$\nu$ is an interior vertex
and~\eqref{assumption_interior_faces} is satisfied.
Then, we have
\[
m^{\rm g} \leq C(\gamma) m^{\rm g}_h,
\qquad\qquad
m^{\rm d} \leq C(\gamma) m^{\rm d}_h.
\]
\end{lem}

\begin{figure}
\centering
\begin{tikzpicture}[scale=3]

\coordinate (a) at (0,0);

\coordinate (a0) at ($(a)+ ( 10:0.9)$);
\coordinate (a1) at ($(a)+ (100:1.1)$);
\coordinate (a2) at ($(a)+ (200:0.8)$);
\coordinate (a3) at ($(a)+ (300:0.7)$);

\coordinate (b0) at ($(a) + ( 60:1.0)$);

\coordinate (a01) at ($(b0)+(  5:0.7)$);
\coordinate (a02) at ($(b0)+( 90:0.7)$);

\coordinate (b1) at ($(a) + (140:0.8)$);

\coordinate (a11) at ($(b1) + (160:0.8)$);

\coordinate (b2) at ($(a) + (-110:1.0)$);

\coordinate (a21) at ($(b2) + (160:0.9)$);
\coordinate (a22) at ($(b2) + (200:0.5)$);
\coordinate (a23) at ($(b2) + (250:1.0)$);
\coordinate (a24) at ($(b2) + (290:0.5)$);
\coordinate (a25) at ($(b2) + (330:0.8)$);

\coordinate (b3) at ($(a) + ( -30:0.7)$);

\coordinate (a31) at ($(b3) + (- 80:0.5)$);
\coordinate (a32) at ($(b3) + (- 10:0.6)$);


\draw[ultra thick] (a) -- (a0);
\draw[ultra thick] (a) -- (a1);
\draw[ultra thick] (a) -- (a2);
\draw[ultra thick] (a) -- (a3);


\draw[dashed] (b0) -- (a);
\draw[dashed] (b0) -- (a0);
\draw[dashed] (b0) -- (a01);
\draw[dashed] (b0) -- (a02);
\draw[dashed] (b0) -- (a1);

\draw[ultra thick] (a0) -- (a01) -- (a02) -- (a1);


\draw[dashed] (b1) -- (a);
\draw[dashed] (b1) -- (a1);
\draw[dashed] (b1) -- (a11);
\draw[dashed] (b1) -- (a2);

\draw[ultra thick] (a1) -- (a11) -- (a2);


\draw[dashed] (b2) -- (a);
\draw[dashed] (b2) -- (a2);
\draw[dashed] (b2) -- (a21);
\draw[dashed] (b2) -- (a22);
\draw[dashed] (b2) -- (a23);
\draw[dashed] (b2) -- (a24);
\draw[dashed] (b2) -- (a25);
\draw[dashed] (b2) -- (a3);

\draw[ultra thick] (a2) -- (a21) -- (a22) -- (a23) -- (a24) -- (a25) -- (a3);


\draw[dashed] (b3) -- (a);
\draw[dashed] (b3) -- (a3);
\draw[dashed] (b3) -- (a31);
\draw[dashed] (b3) -- (a32);
\draw[dashed] (b3) -- (a0);

\draw[ultra thick] (a3) -- (a31) -- (a32) -- (a0);


\draw ($0.33*(a)+0.33*(b0)+0.33*(a0)$) node {1};
\draw ($0.33*(a)+0.33*(b0)+0.33*(a1)$) node {2};
\draw ($0.33*(a)+0.33*(b1)+0.33*(a1)$) node {3};
\draw ($0.33*(a)+0.33*(b1)+0.33*(a2)$) node {4};
\draw ($0.33*(a)+0.33*(b2)+0.33*(a2)$) node {5};
\draw ($0.33*(a)+0.33*(b2)+0.33*(a3)$) node {6};
\draw ($0.33*(a)+0.33*(b3)+0.33*(a3)$) node {7};
\draw ($0.33*(a)+0.33*(b3)+0.33*(a0)$) node {8};

\draw ($0.33*(a0) +0.33*(b0)+0.33*(a01)$) node  {9};
\draw ($0.33*(a01)+0.33*(b0)+0.33*(a02)$) node {10};
\draw ($0.33*(a02)+0.33*(b0)+0.33*(a1) $) node {11};

\draw ($0.33*(a1) +0.33*(b1)+0.33*(a11)$) node {12};
\draw ($0.33*(a11)+0.33*(b1)+0.33*(a2) $) node {13};

\draw ($0.33*(a2) +0.33*(b2)+0.33*(a21)$) node {14};
\draw ($0.33*(a21)+0.33*(b2)+0.33*(a22)$) node {15};
\draw ($0.33*(a22)+0.33*(b2)+0.33*(a23)$) node {16};
\draw ($0.33*(a23)+0.33*(b2)+0.33*(a24)$) node {17};
\draw ($0.33*(a24)+0.33*(b2)+0.33*(a25)$) node {18};
\draw ($0.33*(a25)+0.33*(b2)+0.33*(a3) $) node {19};

\draw ($0.33*(a3) +0.33*(b3)+0.33*(a31)$) node {20};
\draw ($0.33*(a31)+0.33*(b3)+0.33*(a32)$) node {21};
\draw ($0.33*(a32)+0.33*(b3)+0.33*(a0) $) node {22};

\draw[-{Straight Barb[black,line width=1pt,width=6pt,length=4pt]}]%
($(a) + (20:0.2)$) arc (20:355:0.2);
\draw[-{Straight Barb[black,line width=1pt,width=6pt,length=4pt]}]%
($(b0) + (-50:0.2)$) arc (-45:150:0.2);
\draw[-{Straight Barb[black,line width=1pt,width=6pt,length=4pt]}]%
($(b1) + ( 70:0.2)$) arc (70:230:0.2);
\draw[-{Straight Barb[black,line width=1pt,width=6pt,length=4pt]}]%
($(b2) + (140:0.2)$) arc (140:360:0.2);
\draw[-{Straight Barb[black,line width=1pt,width=6pt,length=4pt]}]%
($(b3) + (260:0.2)$) arc (260:400:0.2);

\fill (a) circle (0.05);

\fill (a0) circle (0.05);
\fill (a1) circle (0.05);
\fill (a2) circle (0.05);
\fill (a3) circle (0.05);

\fill (a01) circle (0.05);
\fill (a02) circle (0.05);

\fill (a11) circle (0.05);

\fill (a21) circle (0.05);
\fill (a22) circle (0.05);
\fill (a23) circle (0.05);
\fill (a24) circle (0.05);
\fill (a25) circle (0.05);

\fill (a31) circle (0.05);
\fill (a32) circle (0.05);

\end{tikzpicture}
\caption{Enumeration in an interior vertex patch}
\label{figure_enumeration}
\end{figure}
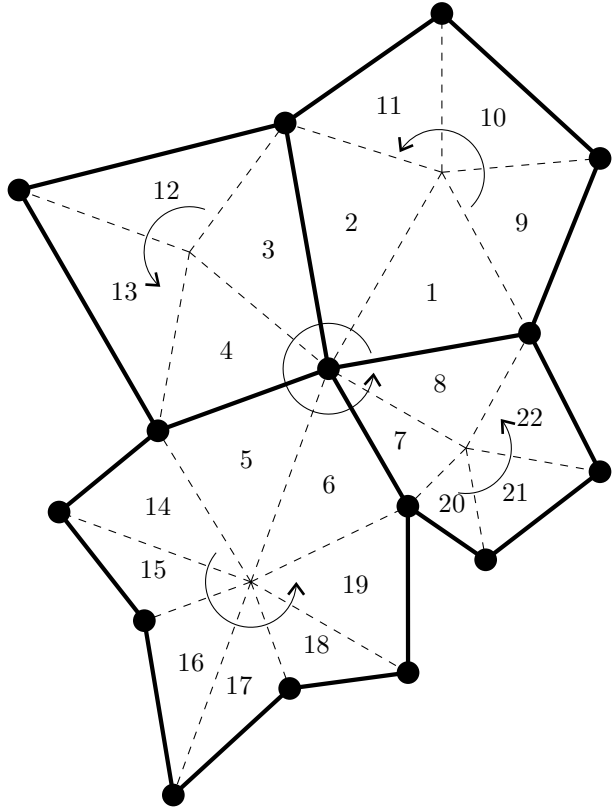

\begin{proof}
For patches consisting of simplices sharing a single vertex,
this result is now standard;
see, e.g., \cite{Braess-Pillwein-Schoeberl:2009,ChaumontFrelet-Vohralik:2022,Ern-Vohralik:2020}.
However, here, as the polygons are further broken down into additional triangles,
the patches we consider do not enter this framework:
some triangles do not have $\nu$ as a vertex.
Nevertheless, a careful read of~\cite{ChaumontFrelet-Vohralik:2022,Ern-Vohralik:2020}
reveals that the proof should still go through for the above patches,
as long as we can provide a suitable ordering (or enumeration) of the triangles in the patch. 

In fact, the properties required from the enumeration have recently been explicitly stated in
\cite[Definition B.1]{Vohralik:2024}.
Under the assumption that a suitable enumeration is available,
the bound $m^{\rm g} \leq C(\gamma) m^{\rm g}_h$
is explicitly established in
\cite[Theorem D.1]{Vohralik:2024}.
The corresponding bound for the divergence-constrained
problem can similarly be obtained following \cite{ChaumontFrelet-Vohralik:2022,Ern-Vohralik:2020} with the aforementioned enumeration,
and the details should appear shortly in~\cite{Demkowicz-Vorhalik:2024}.

Therefore, we have to produce an enumeration of the patch triangles
meeting the requirements of \cite[Definition B.1]{Vohralik:2024},
i.e., list the elements as
$\tautildennu = \{T_1,\dots,T_N\}$ in such a way that:
\begin{itemize}
\item[a)] for $1 < j \leq N$, the triangle $T_j$ shares at least one facet with an already
enumerated triangle,
i.e., there exists $T_\ell$, $1 \leq \ell < j$ such that
$\partial T_j \cap \partial T_\ell \neq \emptyset$;
\item[b)] for $1 \leq j \leq N$, the triangle $T_j$ shares at most two faces with already
enumerated triangles;
\item[c)] for $1 \leq j \leq N$, if $T_j$ shares two faces with, say $T_\ell$ and $T_{\ell'}$,
$1 \leq \ell < \ell' \leq j$, then all the elements sharing the vertex common to $T_j$, $T_\ell$
and $T_\ell'$ are already enumerated.
\end{itemize}

For the standard case where all the triangles share the same vertex,
i.e., as in \cite{Braess-Pillwein-Schoeberl:2009},
such an enumeration can be trivially constructed by looping around the vertex.
Here, the enumeration can be obtained as follows,
see Figure~\ref{figure_enumeration} for an illustration:
first, we enumerate the triangles sharing the vertex~$\nu$ counterclockwise;
we then go through the centroid~$\xbfE$
of the polygonal elements~$\E$ in~$\taunnu$ counterclockwise.
Around each $\xbfE$, there is a closed patch of triangles sharing~$\xbfE$ and covering~$\E$,
and two of them, say $T_{\ell}$ and $T_{\ell+1}$,
have already been enumerated when looping over~$\nu$.
Besides, due to~\eqref{assumption_interior_faces},
the remaining triangles around~$\xbfE$ do not share faces
with any triangle from another polygon~$\E'$ in $\taunnu$.
We can then enumerate those triangles
by completing the loop around~$\xbfE$ counterclockwise,
i.e., going from $T_\ell$ to $T_{\ell+1}$ by crossing faces.
\end{proof}

\begin{remark}
If the triangles~$T_9$ and~$T_{22}$ shared a face in Figure~\ref{figure_enumeration},
then the enumeration would not be valid anymore.
Indeed, all the neighbours of $T_{22}$ would have appeared sooner in the enumeration,
hence violating point c) in the proof above.
This explains why we invoke~\eqref{assumption_interior_faces}
to carry out the proof.
Notwithstanding, (i) the presence of elements~$\E$ sharing multiple faces
does not exclude the existence of a suitable enumeration satisfying \cite[Definition B.1]{Vohralik:2024},
even though\, admittedly, we were not able to provide a generic construction;
moreover, (ii) the lack of a suitable enumeration does not imply that $\p$-robustness does not hold in general.
In other words, the question of $\p$-robustness on general patches is still open.
We deem that there is no severe obstruction to its proof
and that it could be obtained at the price of more technicalities in the
construction of a suitable enumeration.
\end{remark}

Eventually, we deal with the exterior vertices case.

\begin{lem}[$\p$-robust stable minimisation for boundary vertices]
Let Assumption~\ref{assumption_faces} hold true.
Assume that $\nu$ is a boundary vertex and~\eqref{assumption_exterior_faces} is satisfied.
Then, we have
\[
m^{\rm g} \leq C(\gamma) m^{\rm g}_h ,
\qquad\qquad
m^{\rm d} \leq C(\gamma) m^{\rm d}_h.
\]
\end{lem}

\begin{figure}
\centering
\begin{tikzpicture}[scale=3]

\coordinate (a) at (0,0);

\coordinate (a0) at ($(a)+ ( 20:0.9)$);
\coordinate (a1) at ($(a)+ (-40:1.1)$);
\coordinate (a2) at ($(a)+ (-100:0.8)$);
\coordinate (a3) at ($(a)+ (-190:0.7)$);

\coordinate (a01) at ($(a)+ ( -10:1.3)$);

\coordinate (a11) at ($(a)+ ( -60:1.3)$);
\coordinate (a12) at ($(a)+ ( -80:1.3)$);

\coordinate (a21) at ($(a)+ ( -120:1.2)$);
\coordinate (a22) at ($(a)+ ( -130:0.8)$);
\coordinate (a23) at ($(a)+ ( -150:1.0)$);

\draw[ultra thick] (a) -- (a0);
\draw[ultra thick] (a) -- (a1);
\draw[ultra thick] (a) -- (a2);
\draw[ultra thick] (a) -- (a3);

\draw[ultra thick] (a0) -- (a01) -- (a1);
\draw[ultra thick] (a1) -- (a11) -- (a12) -- (a2);

\draw[ultra thick] (a2) -- (a21) -- (a22) -- (a23) -- (a3);

\coordinate (b1) at ($(a)+ ( 70:1.1)$);
\coordinate (b2) at ($(a)+ ( 120:1.0)$);

\coordinate (b01) at ($(a)+ ( 50:1.3)$);

\coordinate (b11) at ($(a)+ ( 85:1.3)$);
\coordinate (b12) at ($(a)+ (105:1.3)$);

\coordinate (b21) at ($(a)+ (130:1.3)$);
\coordinate (b22) at ($(a)+ (150:0.9)$);
\coordinate (b23) at ($(a)+ (160:1.3)$);

\draw[dashed] (a) -- (b1);
\draw[dashed] (a) -- (b2);

\draw[dashed] (a0) -- (b01) -- (b1);
\draw[dashed] (b1) -- (b11) -- (b12) -- (b2);
\draw[dashed] (b2) -- (b21) -- (b22)  -- (b23)-- (a3);

\fill (a) circle (0.05cm);

\fill (a0) circle (0.05cm);
\fill (a1) circle (0.05cm);
\fill (a2) circle (0.05cm);
\fill (a3) circle (0.05cm);

\fill (a01) circle (0.05cm);

\fill (a11) circle (0.05cm);
\fill (a12) circle (0.05cm);

\fill (a21) circle (0.05cm);
\fill (a22) circle (0.05cm);
\fill (a23) circle (0.05cm);

\draw (b1) circle (0.05cm);
\draw (b2) circle (0.05cm);

\draw (b01) circle (0.05cm);

\draw (b11) circle (0.05cm);
\draw (b12) circle (0.05cm);

\draw (b21) circle (0.05cm);
\draw (b22) circle (0.05cm);
\draw (b23) circle (0.05cm);

\end{tikzpicture}
\caption{Symmetrisation of boundary patch}
\label{figure_symmetrisation}
\end{figure}
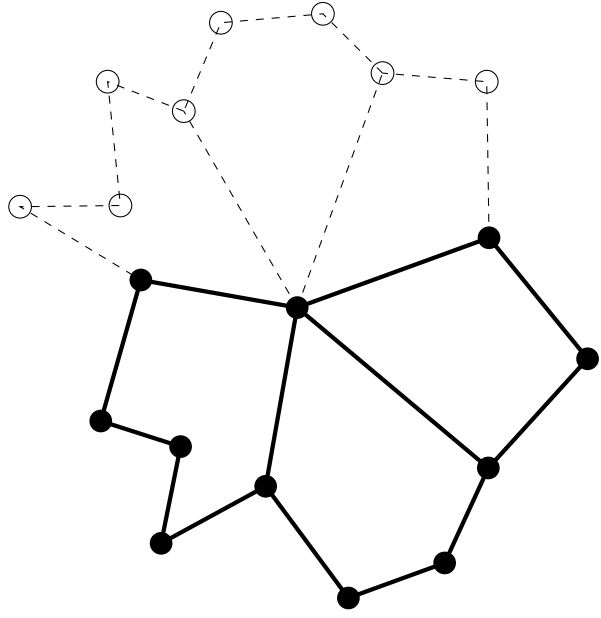

\begin{proof}
We only sketch the main ideas of the proof, since it closely follows that of
\cite[Section 7]{Ern-Vohralik:2020} and \cite[Section 7]{ChaumontFrelet-Vohralik:2022}.
The key idea is to symmetrise the patch around the boundary composed
of the exterior faces sharing the vertex~$\nu$
to obtain an interior patch.
This process is illustrated in Figure~\ref{figure_symmetrisation}.
Display~\eqref{assumption_exterior_faces} ensures that
the boundary actually only consists of two faces,
and the symmetrised patch satisfies~\eqref{assumption_interior_faces}.
The symmetrised patch is an interior patch, for which Lemma~\ref{lemma:discrete_stab_interior} now applies;
the arguments developed in~\cite{ChaumontFrelet-Vohralik:2022,Ern-Vohralik:2020}
can be invoked to show that,
if the discrete stable minimisation property holds true in the symmetrised patch,
then it must hold true in the original patch too.
\end{proof}

Collecting the estimates above, we end up with the following result.

\begin{thm}[$\p$-robust stable minimisation]
Let Assumption~\ref{assumption_faces} hold true.
Assume that Assumption~\ref{assumption_faces} is satisfied. Then, we have
\[
m^{\rm g} \leq C(\gamma) m^{\rm g}_h,
\qquad\qquad
m^{\rm d} \leq C(\gamma) m^{\rm d}_h\,
\]
where the constant~$C(\gamma)$ is independent of~$\p$.
\end{thm}

\end{document}